\documentclass[12pt]{amsart}
\usepackage{amsmath, amssymb}
\newtheorem{lemma}{Lemma}[section]
\newtheorem{theorem}{Theorem}[section]

\newtheorem{proposition}{Proposition}[section]
\newtheorem{conjecture}{Conjecture}[section]
\def\proof{\noindent {\bf Proof. \ }}
\def\1bar{\overline{1}}
\def\2bar{\overline{2}}
\def\d{\partial}
\def\CC{\mathbb{C}}
\def\jbar{\overline{j}}
\def\zbar{\overline{z}}
\def\dbar{\overline{\partial}}

\def\bl{\begin{Lem}}
\def\bl{\begin{Lem}}
\def\el{\end{Lem}}
\def\bp{\begin{Pro}}
\def\ep{\end{Pro}}
\def\bt{\begin{Thm}}
\def\et{\end{Thm}}
\def\bc{\begin{Cor}}
\def\ec{\end{Cor}}
\def\bd{\begin{Def}}
\def\ed{\end{Def}}
\def\br{\begin{remark}}
\def\er{\end{remark}}
\def\be{\begin{Exa}}
\def\ee{\end{Exa}}
\def\bpf{\begin{proof}}
\def\epf{\end{proof}}
\def\ben{\begin{enumerate}}
\def\een{\end{enumerate}}
\def\beq{\begin{equation}}
\def\eeq{\end{equation}}

\def\Re{{\sf Re}\,}
\def\Im{{\sf Im}\,}

\def\phi{\varphi}

\def\d{\partial}

\def\cal{\mathcal}
\def\th{\theta}

\begin{document}

\title{Strong Ramadanov Conjecture for  Real Ellipsoids in $\CC^2$: two approaches }

\author{Jan Gregorovic}
\address{Department of Mathematics, University of Ostrava, Czechia  $\&$ Institute of Discrete Mathematics and Geometry, TU Wien, Vienna, Austria}
\email{Jan.Gregorovic@seznam.cz}

\author{Ilya Kossovskiy}
\address{Department of Mathematics, Sustech University, Shenzhen, China  \& International Center of Mathematics, Shenzhen, China}
\email{ilyakos@sustech.edu.cn}

\author{Song-Ying Li}
\address{Department of Mathematics, University of California, Irvine, USA}
\email{sli@uci.edu }

\author{Kevin Shi}
\address{Tarbut $V^{\prime}$Torah  Community Day School, USA}
\email{kevinshi09102@gmail.com}

\begin{abstract}
In this paper, we provide two (significantly different) proofs of the well known Strong Ramadanov Conjecture for the class of real ellipsoids in the complex $2$-space.    
\end{abstract}

\maketitle

\section{Introduction}
Let $D$ be a bounded pseudoconvex domain in $\CC^n$ with smooth boundary. 
 The Fefferman type complex Monge-Amp\'ere equation or Fefferman equation on $D$ is given by
 \begin{equation}\label{F}
 J(\gamma):= -\det
\begin{pmatrix}
\gamma & \partial \gamma \\
(\partial \gamma)^* & H(\gamma)
\end{pmatrix} = 1 \quad \hbox{in } D; \quad \gamma=0 \ \hbox{ on } \d D.
\end{equation}
It was  proved by Fefferman \cite{Fe76} that if $\rho$ with $u=-\log(-\rho)$ being strictly plurisubharmonic  is a solution of (\ref{F}), then it is unique. 
Here, $H(\gamma)$ denotes the complex Hessian matrix of $\gamma$. The existence of the solution was proved by Cheng and Yau
 in \cite{CY80} which gave the existence and uniqueness of K\"ahler-Einstein metric $\sum_{i,j=1}^n u_{i\jbar} dz_i \otimes d\zbar_j$ on $D$. When $D$ is also strictly 
 pseudoconvex, Cheng and Yau proved that $\rho\in C^{n+3/2} (\overline{D})$. The asymptotic expansion formula for the solution $\rho$ given by Lee and Melrose \cite{LM82} is as follows:
 \begin{equation}\label{LM}
 \rho(z)=\rho_0 \sum_{j=0}^\infty c_i ((-\rho_0^{n+1} \log(-\rho_0))^j
 \end{equation}
 where $\rho_0$ is a smooth strictly plurisubharmonic defining function of $D$, with $c_j\in C^\infty (\overline{D})$ and $c_0>0$ on $\d D$.  The equation (\ref{LM}) implies that
 $\rho\not\in C^{n+2}(\overline{D})$ in general, unless $c_1\equiv 0$. Accordingly, the coefficient $c_1$ in the expansion \eqref{LM} (which is the obstruction for having $\rho\in C^{n+2}(\overline{D})$)  is called {\em the obstruction function}. This function vanishes in the case when $D$ is a {\em ball} in $\mathbb C^n$.  This leads to the so-called {\em Strong Ramadanov Conjecture} or {\em obstruction flatness Conjecture}:
 \begin{conjecture} Let $D$ be a bounded strictly pseudoconvex domain in $\CC^n$ with $C^\infty$ boundary. If the Fefferman potential function $\rho \in C^{n+2} (\overline{D})$, then $\d D$ is locally spherical. 
 \end{conjecture} 
 
 The conjecture has been studied by many authors. Curry and Ebenfelt in \cite{CE19} and \cite{CE21} gave some sufficient conditions so that the Strong Ramadanov Conjecture holds, in particular
 for the domain in $\CC^2$ has tranverse symmetry, using CR geometry and Chern-Moser normal form. Several counterexamples for domains in complex manifolds were given by Ebenfelt, Xiao and Xu in \cite{EXX22} and \cite{EXX25}. However, the Strong Ramadanov Conjecture for a smoothly bounded strictly pseudoconvex domain in $\mathbb C^n$ is widely open. Even if $D$ is a  real ellipsoid in $\CC^2$. 

We note that the Strong Ramadanov Conjecture implies the (regular) Ramadanov Conjecture on the sphericity of a domain with the vanishing to infinite order at the boundary log-term in its Fefferman's Bergman kernel expansion. Ramdanov Conjecture has been settled in the $\CC^2$ case: see \cite{BdM,Gr87b} and the discussion in \cite{EXX25}.
 
 In \cite{Fe76}, Fefferman gave  an approximation procedure for a solution of \eqref{F} as follows.  Take any  strictly plurisubharmonic defining function $\rho_0\in C^\infty(\overline{D})$  of $D$. Define
 \begin{equation}
 \rho_1=\rho_0 J(\rho_0)^{-1/(n+1)}.
 \end{equation}
Then 
$
J(\rho_1)=1+E_1(-\rho_1).
$
If $J(\rho_j)=1+E_j(-\rho_j)^j$, define
\begin{equation}
\rho_{j+1}= \rho_{j}\Big(1- {E_{j} (-\rho_{j})^{j}\over (n+1-j)(j+1) }\Big),\quad 1\le j\le n.
\end{equation}
Then
\begin{equation}
J(\rho_{j+1})=1+E_{j+1} (- \rho_{j+1} )^{j+1} ,\quad 1\le j\le n.
\end{equation}
If $E_{n+1}\not\equiv 0$ on $\d D$, one cannot find a smooth function $E_{n+2}$ and a constant $c$ such that $\rho_{n+2}=\rho_{n+1}(1+E_{n+2}(-\rho_{n+1})^{n+1})$
such that $J(\rho_{n+2})=1+O((-\rho_{n+1})^{n+2})$. Instead, one can find a smooth function $E$ such that 
if $\rho_{n+2}=\rho_{n+1}(1+E(-\rho_{n+1})^{n+2} \log(-\rho_{n+1}))$, then $J(\rho_{n+2})=1+O((-\rho_{n+1})^{n+2} \log(-\rho_{n+1}))$. Continuing this process, one can get Lee and
Melrose's asymptotic expansion. Therefore, if the Fefferman potential function $\rho\in C^\infty(\overline{D})$ or $\rho\in C^{n+2}(\overline{D})$, then 
\begin{equation}\label{theA}
{\cal A}=0 \ \ \hbox{on }\ \d D \quad \hbox{with} \quad {\cal A}:=(J(\rho_{n+1})-1)(-\rho_{n+1})^{-n-1}.
\end{equation}
The function $\cal A$ in \eqref{theA} is (up to the constant factor $(n+1)$) precisely the {\em obstruction function} from \eqref{LM}.

As the first main outcome of the current paper, by developing the original Fefferman's approximation procedure, we compute an explicit formula for the obstruction function ${\cal A}$ in terms of a smooth  plurisubharmonic defining function $\gamma$ of $D\subset \CC^2$. As an application, 
we prove the following theorem.
\begin{theorem}\label{main} Let $D$ be a real ellipsoid in $\CC^2$. If the Fefferman potential function $\rho_D$ of $D$ belongs to $C^4(\overline{D})$, then $D=B_2$, the unit ball in $\CC^2$.
\end{theorem} 
Consequently, the Strong Ramadanov Conjecture  holds for real ellipsoids in $\CC^2$. We note that for ellipsoids sufficiently close to the ball, the result also follows from earlier work of Hirachi \cite{Hi93} and Curry-Ebenfelt \cite{CE19}.

As the second main outcome, we introduce a new technique for computing the obstruction function $\mathcal A$ {\em directly in terms of the defining function $\gamma$}. The approach is based on applying the technique of associated differential equations, used by the first two authors in a number of recent publications (e.g.\cite{divergence,nonminimalODE,nonanalytic,hjy}), and using subsequently the theory of {\em Tresse differential invariants for second order ODEs}. The resulting expression for $\mathcal A$ is cumbersome (which is natural to expect as it involves the $8$-jet of the defining function $\gamma$), however, in the case of real ellipsoids it can be analyzed directly and we show that its identical vanishing implies the sphericity of the ellipsoid. This gives a second, alternative proof of Theorem \ref{main} and The Strong Ramadanov Conjecture for ellipsoids. 

\section{Ramadanov Conjecture via Fefferman iterations}

\subsection{Preliminary and formula for ${\cal A}$}

Let $D$ a bounded strictly pseudoconvex domain in $\mathbb{C}^n$ with smooth boundary $\d D$.
Let $\gamma\in C^\infty(\overline{D})$ be a strictly  plurisubharmonic  defining function for $D$.
Then
\begin{eqnarray*}
J(\gamma)
&= &-\det
\begin{pmatrix}
\gamma & \partial \gamma \\
(\partial \gamma)^* & H(\gamma)
\end{pmatrix}\\
&=& (-\gamma)\det
\begin{pmatrix}
\gamma & \partial \gamma \\
0 & H(\gamma) - \frac{1}{\gamma}(\partial \gamma)^*\partial \gamma
\end{pmatrix}\\
&=& (-\gamma)\det\Big[ H(\gamma) + \frac{1}{-\gamma}\partial\gamma \otimes \bar\partial\gamma \Big].
\end{eqnarray*}
Here,  again,$H(\gamma)$ denotes the complex Hessian matrix of $\gamma$. 

{\bf For the case  $n=2$}, one has
\begin{eqnarray*}
J(\gamma)
&= &(-\gamma)\Big[
(\gamma_{1\bar1} + \tfrac{1}{-\gamma}|\gamma_1|^2)
(\gamma_{2\bar2} + \tfrac{1}{-\gamma}|\gamma_2|^2)  -
(\gamma_{1\bar2} + \tfrac{1}{-\gamma}\gamma_1\bar\gamma_2)
(\gamma_{2\bar1} + \tfrac{1}{-\gamma}\gamma_2\bar\gamma_1)
\Big]\\
&=& (-\gamma)\Big[
\gamma_{1\bar1}\gamma_{2\bar2}
+ \frac{1}{-\gamma}\gamma_{1\bar1}|\gamma_2|^2
+ \frac{1}{-\gamma}\gamma_{2\bar2}|\gamma_1|^2
- \gamma_{1\bar2}\gamma_{2\bar1}
 - \frac{1}{-\gamma}\gamma_{1\bar2}\gamma_2\bar\gamma_1
- \frac{1}{-\gamma}\gamma_{2\bar1}\gamma_1\bar\gamma_2
\Big]\\
&=& (-\gamma)\Big[
\det H(\gamma)+ \frac{1}{-\gamma}\big(
\gamma_{1\bar1}|\gamma_2|^2 + \gamma_{2\bar2}|\gamma_1|^2
\big)
- \frac{1}{-\gamma}\big(
\gamma_{1\bar2}\gamma_2\bar\gamma_1 + \gamma_{2\bar1}\gamma_1\bar\gamma_2
\big)
\Big]\\
&=& (-\gamma)\det H(\gamma)
+ \big[
\gamma_{1\bar1}|\gamma_2|^2 + \gamma_{2\bar2}|\gamma_1|^2
- \gamma_{1\bar2}\gamma_2\bar\gamma_1
- \gamma_{2\bar1}\gamma_1\bar\gamma_2
\big].
\end{eqnarray*}
For any Hermitian matrices $M$ and $N$ be $2\times 2$, we define
\begin{equation}
{\cal L}_N(M)=N_{1\bar1} M_{2\bar 2} +N_{2\bar2} M_{1\bar1} -N_{1\bar2} M_{2, \bar1}-N_{2\bar1} M_{1\bar2}.
\end{equation}
Then
\begin{equation}
{\cal L}_N(M)={\cal L}_M (N)\quad \hbox{and} \quad {\cal L}_M(M)=2\det M.
\end{equation}
This implies that
\begin{proposition} For any real-valued smooth function $g$ on a domain in $ \mathbb{C}^2$, one has
\begin{equation}\label{J1}
J(g)
= -g\det H(g) + {\cal L}_{H(g)}(\d g\times \dbar g).
\end{equation}
\end{proposition}

Let $f$ and $g$ be two real-valued functions. We define a $2\times 2$ Hermitian matrix
\begin{equation}
D(f, g)=\Big[ f_i g_{\jbar}+g_i f_{\jbar}\Big]=\d f\otimes \dbar g+\d g\otimes \dbar f.
\end{equation}
Then
\begin{equation}
D(f, g)=D(g,f) \ \hbox{ and }\ H(f g)=g H(f)+D(f, g)+f H(g).
\end{equation}
Then
\begin{eqnarray*}
H_0(fg):&=&H(fg)-{1\over fg} (\d(fg)\otimes \dbar(fg))\\
&=&fH(g)+gH(f)+D(f,g)\\
&&-{1\over fg}\Big(g^2(\d f\otimes \dbar f)+fg(\d f\otimes \dbar g+\d g \otimes \dbar f)+f^2(\d g\otimes \dbar g)\Big)\\
&=&fH(g)+gH(f)
-{g\over f} ( \d f\otimes \dbar f)-{f\over g} (\d g\otimes \dbar g).
\end{eqnarray*}
Then
\begin{eqnarray*}
\lefteqn{2\det H_0(fg)}\\
&=&{\cal L}_{fH(g)+gH(f)
-{g\over f} ( \d f\otimes \dbar f)-{f\over g} (\d g\otimes \dbar g)} \Big( fH(g)+gH(f)
-{g\over f} ( \d f\otimes \dbar f)-{f\over g} (\d g\otimes \dbar 9)\Big)\\
&=&f {\cal L}_{H(g)} \Big( fH(g)+gH(f)
-{g\over f} ( \d f\otimes \dbar f)-{f\over g} (\d g\otimes \dbar g)\Big)\\
&&+g {\cal L}_{H(f)} \Big( fH(g)+gH(f)
-{g\over f} ( \d f\otimes \dbar f)-{f\over g} (\d g\otimes \dbar g)\Big) \\
&&-{g \over f} {\cal L}_{( \d f\otimes \dbar f)} \Big( fH(g)+gH(f)
-{g \over f} ( \d f\otimes \dbar f)-{f\over g} (\d g\otimes \dbar g)\Big)\\
&&-{f \over g} {\cal L}_{ (\d g\otimes \dbar g)} \Big( fH(g)+gH(f)
-{g \over f} ( \d f\otimes \dbar f)-{f\over g} (\d g\otimes \dbar g)\Big)\\
&=&f^2 2\det H(g) +2fg {\cal L}_{H(g)}(H(f))-2g{\cal L}_{H(g)}(\d f\otimes \dbar f)-2{f^2\over g} {\cal L}_{H(g)}(\d g\otimes \dbar g)\\
&&+g^2 2 \det H(f)-2{g^2\over f} {\cal L}_{H(f)}(\d f\otimes \dbar f)-2f{\cal L}_{H(f)}(\d g \otimes \dbar g)\\
&&+2{\cal L}_{\d f\otimes \dbar f}(\d g\otimes \dbar g)\\
&=&-2{f^2\over g}\Big( {\cal L}_{H(g)}(\d g\otimes \dbar g) -g\det H(g) \Big)
-2{g^2\over f} \Big({\cal L}_{H(f)}(\d f\otimes \dbar f)-f \det H(f)\Big)\\
&&+2fg {\cal L}_{H(g)}(H(f))-2g{\cal L}_{H(g)}(\d f\otimes \dbar f)-2f{\cal L}_{H(f)}(\d g \otimes \dbar g)
+2{\cal L}_{\d f\otimes \dbar f}(\d g\otimes \dbar g).
\end{eqnarray*}
Therefore,
dividing the both sides by $1/2$, using the fact that
$$
J(fg)=-(fg) \det H_0(fg),
$$
and performing some simplifications, we have proved the following theorem.
\begin{theorem}\label{1.1} Let $f$ and $g$ be two functions on a domain in $\mathbb{C}^2$. Then
\begin{eqnarray}\label{1}
J(fg)
&=& f^3 J(g) +g^3 J(f) 
-g^2f^2 {\cal L}_{H(f)}(H(g))+B(f, g),
\end{eqnarray}
where
\begin{eqnarray}
B(f, g)&=&fg^2 {\cal L}_{H(g)}(\d f\otimes \dbar f)+f^2 g  {\cal L}_{H(f)}(\d g\otimes \dbar g) \nonumber\\
&&-fg {\cal L}_{\d f\otimes \dbar f}(\d g\otimes \dbar g).
\end{eqnarray}
\end{theorem}

Next, notice that
\begin{eqnarray}
{\cal L}_{\d f\otimes \dbar f}(\d g\otimes \dbar g)
&=&|f_1|^2 |g_2|^2+|f_2|^2 |g_1|^2-2\hbox{Re} f_1 f_{\bar 2} g_2 g_{\bar1}\nonumber\\
&=&|f_1 g_2-f_2 g_1|^2.
\end{eqnarray}
Now we will take
\begin{equation}
f(z)=J(\gamma)^{-1/3},\quad g(z)=\gamma(1+E(-\gamma)+F \gamma^2).
\end{equation}
Then
$$
\d f\otimes \dbar f={1\over 9 J^{8/3}} \d J\otimes \dbar J,
$$
\begin{eqnarray}
fg {\cal L}_{\d f\otimes \dbar f}(\d g\otimes \dbar g)
&=&{g\over 9 J^3} |J_1 g_2-J_2 g_1|^2
\end{eqnarray}
and
$$
H(f)=-{1\over 3 J^{4/3}} (H(J)-{4\over 3J}\d J\otimes \dbar J) .
$$
Then
\begin{eqnarray} \label{f}
f H(f)-\d f\otimes \dbar f
&=-&{1\over 3 J^{5/3}}(H(J)-{4\over 3 J}\d J\otimes \dbar J)-{1\over 9 J^{8/3}} \d J\otimes \dbar J\nonumber\\
&=&-{1\over 3 J^{5/3}}H(J)+{1\over 3 J^{8/3}}\d J\otimes \dbar J
\end{eqnarray}
and
\begin{eqnarray*}
\lefteqn{f^2 g{\cal L}_{H(f)}(\d g\otimes \dbar g)-fg{\cal L}_{\d f\otimes \dbar f}(\d g\otimes \dbar g)}\\
&=& f g {\cal L}_{ f H(f)-\d f\otimes \dbar f}(\d g\otimes \dbar g)\\
&=&-{g\over 3 J^2} {\cal L}_{H(J)} (\d g\otimes \dbar g)+{g \over 3J^3} {\cal L}_{\d J\otimes \dbar J}(\d g\otimes \dbar g).
\end{eqnarray*}
Moreover,
\begin{eqnarray*}
\lefteqn{-f^2 g^2 {\cal L}_{H(f)}(H(g))+f g^2{\cal L}_{H(g)}(\d f\otimes\dbar f)}\\
&=&-f g^2 \Big({\cal L}_{H(g)}(f H(f))- {\cal L}_{H(g)}(\d f\otimes \dbar f))\\
&=&  g^2 \Big({1\over 3J^2} {\cal L}_{H(g)}(H(J))-{1\over 3 J^3} {\cal L}_{H(g)}(\d J\otimes \dbar J)\Big).
\end{eqnarray*}
Therefore,
\begin{eqnarray*}
\lefteqn{B(f, g)}\\
&=&-f^2 g^2 {\cal L}_{H(f)}(H(g))+f g^2{\cal L}_{H(g)}(\d f\otimes\dbar f)+f^2 g{\cal L}_{H(f)}(\d g\otimes \dbar g)-fg{\cal L}_{\d f\otimes \dbar f}(\d g\otimes \dbar g)\\
&=&  {g^2\over 3J^2}  \Big[{\cal L}_{H(g)}(H(J))-{1\over J} {\cal L}_{H(g)}(\d J\otimes \dbar J)\Big]-{g\over 3 J^2}
\Big[ {\cal L}_{H(J)} (\d g\otimes \dbar g)-{1 \over J} {\cal L}_{\d J\otimes \dbar J}(\d g\otimes \dbar g)\Big].
\end{eqnarray*}
By (\ref{f}), one has
\begin{eqnarray*}
J(f)&=&-f \det H_0(f)=-f \det \Big(-{1\over 3 J^{4/3}} H(J)+{1\over 3 J^{7/3}} \d J \otimes \dbar J\Big)\\
&=&-{1\over  9J^3} \det \Big(H(J)-{1\over J} \d J\otimes \dbar J\Big).
\end{eqnarray*}
This gives
\begin{proposition} Let $f=J^{-1/3}$. Then
\begin{eqnarray}
J(f) 
&=&-{1\over  9J^3} \det \Big(H(J)-{1 \over J} \d J\otimes \dbar J\Big).
\end{eqnarray}
\end{proposition}

Now, by Theorem \ref{1.1}, with $J=J(\gamma)$, one has
\begin{eqnarray}\label{J(fg)}
J(f g)&=&{1\over J} J(g) +g^3 J(f)+B(f, g) \nonumber\\
&=&{J(g) \over J} -{g^3\over 9 J^3} \det \Big(H(J)-{1\over J}\d J\otimes \dbar J\Big) \nonumber \\
&&+ {g^2\over 3J^2}  \Big({\cal L}_{H(g)}(H(J))-{1\over J} {\cal L}_{H(g)}(\d J\otimes \dbar J)\Big) \nonumber\\
&&-{g\over 3 J^2}
\Big( {\cal L}_{H(J)} (\d g\otimes \dbar g)-{1 \over J} {\cal L}_{\d J\otimes \dbar J}(\d g\otimes \dbar g)\Big) \nonumber\\
&=&{J(g) \over J} -{g^3\over 9 J^3} \det \Big(H(J)-{1\over J}\d J\otimes \dbar J\Big) \nonumber \\
&&- {g\over 3J^2}  {\cal L}_{H(J)}( -g H(g)+\d g\otimes \dbar g) \nonumber\\
&&+{g\over 3 J^3}
{\cal L}_{\d J\otimes \dbar J}(\d g\otimes \dbar g -g H(g)).
\end{eqnarray}

Next define
\begin{eqnarray}
g(z)&=&\gamma(1+E (-\gamma)+F\gamma^2).
\end{eqnarray}
Then
\begin{equation}
g_i=a \gamma_i- (-\gamma)^2(E_i +F_i (-\gamma)),
\end{equation}
where
\begin{eqnarray}
a=(1+2E(-\gamma) +3F\gamma^2).
\end{eqnarray}
We shall prove
\begin{proposition} 
\begin{eqnarray}
\qquad \d g\otimes \dbar g
&=&a^2  \d \gamma \otimes \dbar \gamma -(\gamma)^2  D(E,\gamma)\nonumber\\
&&\quad  -(-\gamma)^3 ( 2E D(E, \gamma)+ D(F, \gamma) )+\gamma^4 A_1.
\end{eqnarray}
\end{proposition}
\proof  Since $a=1+2E(-\gamma)+3F\gamma^2$, one has
\begin{eqnarray*}
\d g \otimes \dbar g &=& (a \d \gamma -(-\gamma)^2(\d E, +(-\gamma) \d F) (a\dbar \gamma -(\gamma)^2 (\dbar E+(-\gamma) \dbar F)\\
&=&a^2  \d \gamma \otimes \dbar \gamma -(\gamma)^2 a (D(E, \gamma)+(-\gamma) D(F, \gamma))\\
&&+\gamma^4 (\d E+(-\gamma) \d F \otimes (\dbar E+(-\gamma)\dbar F)\\
&=&a^2 \d \gamma \otimes \dbar \gamma 
-(\gamma)^2  D(E,\gamma) -(-\gamma)^3 D(F, \gamma) -(-\gamma)^3 2E D(E, \gamma) +\gamma^4 A_1.
\end{eqnarray*}
This gives the proof of the proposition. 

Since
\begin{eqnarray}
g_{i\jbar}&=&(\gamma- E \gamma^2+F \gamma^3) _{i\jbar} \nonumber\\
&=& (a \gamma_i -(-\gamma)^2 (E_i+(-\gamma) F_i)_{\jbar} \nonumber\\
&=&  a\gamma_{i\jbar}
 +\gamma_i a_{\jbar} +(-\gamma) (2 E_i +3(-\gamma) F_i) \gamma_{\jbar}
-(-\gamma)^2 (E_{i\jbar} +(-\gamma) F_{i\bar j} -F_i \gamma_{\jbar} ) \nonumber\\
&=& a \gamma_{i\jbar} -(-\gamma)^2 (E_{i\jbar}+(-\gamma) F_{i\jbar}) +(-\gamma) \Big(2 E_i \gamma_{\jbar}+(-\gamma) 3F_i \gamma_{\jbar}\Big)\nonumber\\
&&+ \gamma_i ( 2E_{\jbar} (-\gamma) -2E\gamma_{\jbar}+F_{\jbar} 3(-\gamma)^2 -6(-\gamma) F \gamma_{\jbar})\nonumber\\
&=& a \gamma_{i\jbar} -(-\gamma)^2( E_{i\jbar}+(-\gamma) F_{i\jbar})\nonumber\\
&& +(-\gamma) \Big(2D(E, \gamma)+3(-\gamma) D(F, \gamma)\Big)
-(2E +6(-\gamma)F) \gamma_i \gamma_{\jbar}.\nonumber
\end{eqnarray}
Therefore,
\begin{eqnarray}\label{For1}
H(g)&=&a H(\gamma)-\Big( 2E+6F(-\gamma)\Big)  \d \gamma \otimes \dbar \gamma\nonumber
\\
&&
+2 (-\gamma) D(E, \gamma)+3 \gamma^2 D(F, \gamma)
-  \gamma^2 H(E) -(-\gamma)^3 H(F).\quad
\end{eqnarray}
Notice that
\begin{eqnarray*}
a g&=&\gamma (1+E(-\gamma)+F\gamma^2) (1+2E(-\gamma)+3F\gamma^2)\\
&=&\gamma(1+3 E(-\gamma)+(2E^2+4F)\gamma^2+5EF (-\gamma)^3+3F^2 \gamma^4),
\end{eqnarray*}
\begin{eqnarray*}
\lefteqn{g (E+3F(-\gamma))+a^2}\\
&=&-(-\gamma) (1+E (-\gamma)+F\gamma^2)(E+3F(-\gamma))+(1+2E(-\gamma)+3F\gamma^2)^2\\
&=&1+2(2E(-\gamma)+3F\gamma^2)+(-\gamma)^2(2E+3F(-\gamma))^2\\
&&-(-\gamma)\Big(E+3F(-\gamma) +E^2(-\gamma) +3E F(-\gamma)^2+EF (-\gamma)^2+3F^2(-\gamma)^3\Big)\\
&=&1+3E(-\gamma)+3(F+E^2) \gamma^2 +8EF (-\gamma)^3+6F^2\gamma^4
\end{eqnarray*}
and
$$
-g (-\gamma)-\gamma^2-2E (-\gamma)^3= -E (-\gamma)^3+F(-\gamma)^4.
$$
Then
\begin{eqnarray*}
\lefteqn{{1\over 2} (-g) H(g)+\d g \otimes \dbar g}\\
&=&-{g \over 2} a H(\gamma)+g \Big( E+3F(-\gamma)\Big)  \d \gamma \otimes \dbar \gamma\nonumber
\\
&&
-g (-\gamma) D(E, \gamma)-{3  g\over 2} \gamma^2 D(F, \gamma)
+{g \over 2}  \gamma^2 H(E) +{g\over 2} (-\gamma)^3 H(F)\\
&&+a^2 \d \gamma \otimes \dbar \gamma - (\gamma)^2  D(E,\gamma) -(-\gamma)^3 ( 2E D(E, \gamma)+ D(F, \gamma) )+\gamma^4 A_1\\
&=&{(-\gamma)  \over 2} (1+3 E(-\gamma) +(4F+2E^2)  \gamma^2) H(\gamma)\nonumber
\\
&&
-(-\gamma)^3 E D(E, \gamma)+{  (-\gamma) ^3 \over 2} D(F, \gamma)
-{(-\gamma)^3 \over 2}  H(E) \\
&&+\Big(1+3E(-\gamma)+3(F+E^2) \gamma^2 +8EF (-\gamma)^3\Big)\d \gamma \otimes \dbar \gamma +\gamma^4 A_2\\
&=&{(-\gamma)  \over 2} b  H(\gamma)\nonumber
-(-\gamma)^3 E D(E, \gamma)+{  (-\gamma) ^3 \over 2} D(F, \gamma)
-{(-\gamma)^3 \over 2}  H(E) +c \d \gamma \otimes \dbar \gamma +\gamma^4 A_2,
\end{eqnarray*}
where
\begin{equation}
b= 1+3 E(-\gamma) +(4F+2E^2)  \gamma^2
\end{equation}
and
\begin{equation}
c=1+3E(-\gamma)+3(F+E^2) \gamma^2 +8EF (-\gamma)^3.
\end{equation}
Then
\begin{equation}
c=b+(E^2-F)\gamma^2+8EF (-\gamma)^3
\end{equation}
Notice that
$$
{\cal L}_{\d\gamma\otimes \dbar \gamma}(\d \gamma\otimes \dbar \gamma)=0,
\quad {\cal L}_{\d \gamma\otimes \dbar \gamma}(D(E, \gamma))= {\cal L}_{\d \gamma\otimes \dbar \gamma}(D(F, \gamma))=0,
$$

Then
\begin{eqnarray*}
J(g)&=&-g \det H(g)+{\cal L}_{H(g)}(\d g\otimes \dbar g)\\
&=&{\cal L}_{H(g)}(-{g \over 2} H(g)+\d g\otimes \dbar g)\\
&=&{(-\gamma)  \over 2} b {\cal L}_{H(g)}(H(\gamma))\nonumber
\\
&&
-(-\gamma)^3 E{\cal L}_{H(g)} D(E, \gamma)+{  (-\gamma) ^3 \over 2} {\cal L}_{H(g)} D(F, \gamma)
-{(-\gamma)^3 \over 2}  {\cal L}_{H(g)} (H(E)) \\
&&+c  {\cal L}_{H(g)} (\d \gamma \otimes \dbar \gamma) +\gamma^4 {\cal L}_{H(g)}(A_2)
\\
&=&{(-\gamma)  \over 2} b  a {\cal L}_{H(\gamma)}(H(\gamma))\nonumber
-2{(-\gamma)\over 2} b (E+3F(-\gamma)) {\cal L}_{H(\gamma)}( \d \gamma \otimes\dbar \gamma) \nonumber\\
&&+b {(-\gamma)\over 2}2(-\gamma) {\cal L}_{H(\gamma)}(D(E, \gamma)) +b{(-\gamma)^3\over 2}\Big(3 {\cal L}_{H(\gamma)}(D(F,\gamma))-{\cal L}_{H(\gamma)}(H(E)\Big)\\
&&
-(-\gamma)^3 E {\cal L}_{H(\gamma)} D(E, \gamma)+{  (-\gamma) ^3 \over 2} {\cal L}_{H(\gamma)} (D(F, \gamma))
-{(-\gamma)^3 \over 2}  {\cal L}_{H(\gamma)-2E \d \gamma \otimes \dbar \gamma } (H(E)) \\
&&+ a (b+(E^2-F)\gamma^2+8EF(-\gamma)^3)  {\cal L}_{H(\gamma)} (\d \gamma \otimes \dbar \gamma)\\
&&-(1+3E(-\gamma)) \gamma^2 {\cal L}_{H(E)}(\d \gamma\otimes \dbar \gamma)-(-\gamma)^3 {\cal L}_{H(F)}(\d \gamma \otimes \dbar \gamma) +O(\gamma^4 ).
\end{eqnarray*}

Notice that 
\begin{eqnarray*}
\lefteqn{-{(-\gamma)^3\over 2} {\cal L}_{H(\gamma)-2E\d \gamma\otimes \dbar \gamma } (H(E))
+(1+3E(-\gamma) )\gamma^2{\cal L}_{\d \gamma \otimes\dbar \gamma}(H(E))}\\
&=&-{(-\gamma)^3\over 2} {\cal L}_{H(\gamma)}(H(E))+(1+2E(-\gamma))\gamma^2 {\cal L}_{\d \gamma \otimes\dbar \gamma}(H(E))
\end{eqnarray*}
and
$$
ab {(-\gamma) \over 2} {\cal L}_{H(\gamma)}(H(\gamma))+ab {\cal L}_{H(\gamma)}(\d \gamma\otimes
\dbar \gamma)
=ab J(\gamma).
$$
Therefore,
\begin{eqnarray*}
\lefteqn{J(g)}\\&=& b  a J(\gamma)+ \Big[a ((E^2-F)\gamma^2+8EF(-\gamma)^3)  
-(-\gamma) b (E+3F(-\gamma))\Big) (J-(-\gamma) \det H(\gamma))  \nonumber\\
&&+(1+3E(-\gamma))  \gamma^2 {\cal L}_{H(\gamma)}(D(E, \gamma)) +{(-\gamma)^3\over 2}\Big(3 {\cal L}_{H(\gamma)}(D(F,\gamma))-{\cal L}_{H(\gamma)}(H(E)\Big]\\
&&
-(-\gamma)^3 {\cal L}_{H(\gamma)} D(E, \gamma)+{  (-\gamma) ^3 \over 2} {\cal L}_{H(\gamma)} (D(F, \gamma))
-{(-\gamma)^3 \over 2}  {\cal L}_{H(\gamma)} (H(E)) \\
&&-(1+2E(-\gamma)) \gamma^2 {\cal L}_{H(E)}(\d \gamma\otimes \dbar \gamma)-(-\gamma)^3 {\cal L}_{H(F)}(\d \gamma \otimes \dbar \gamma) +O(\gamma^4 ).
\end{eqnarray*}
Notice that
\begin{eqnarray*}
ab&=&(1+2E(-\gamma)+3F\gamma^2)(1+3E(-\gamma)+(4F+2E^2)\gamma^2)\\
&=&1+5E(-\gamma)+(7F+8E^2)\gamma^2+E(17F+4E^2)(-\gamma)^3+O(\gamma^4)
\end{eqnarray*}
and
\begin{eqnarray*}
\lefteqn{ a ((E^2-F)\gamma^2+8EF(-\gamma)^3)   -(-\gamma) b (E+3F(-\gamma)) }\\
&=&-bE (-\gamma)+a(E^2-F)\gamma^2-3bF \gamma^2+8 EF(-\gamma)^3+O(\gamma^4) \\
&=&-(1+3E(-\gamma)+(4F+2E^2)\gamma^2) E (-\gamma)+(1+2E(-\gamma))(E^2-F)\gamma^2\\
&&-3(1+3E(-\gamma)) F \gamma^2+8 EF(-\gamma)^3+O(\gamma^4) \\
&=&-E(-\gamma)-2(E^2 +2F)\gamma^2-7EF(-\gamma)^3+O(\gamma^4).
\end{eqnarray*}
Therefore,
\begin{eqnarray*}
J(g)&=&\Big( b  a - \Big(E(-\gamma)+2(E^2+2F)\gamma^2+7EF(-\gamma)^3\Big) J(\gamma)\\
&&+ (-\gamma) \Big(E(-\gamma)+2(E^2+2F)\gamma^2+7EF(-\gamma)^3\Big)\det H(\gamma) \\
&&+\Big(\gamma^2)+(-\gamma)^3 2E) \Big){\cal L}_{H(\gamma)} (D(E, \gamma))+2  (-\gamma) ^3 {\cal L}_{H(\gamma)} (D(F, \gamma))
 \\
&&-(1+2E(-\gamma)) \gamma^2 {\cal L}_{H(E)}(\d \gamma\otimes \dbar \gamma)-(-\gamma)^3 {\cal L}_{H(F)}(\d \gamma \otimes \dbar \gamma) +O(\gamma^4 )\\
&=&\Big( 1+4E(-\gamma)+ 3(2E^2+F)\gamma^2+E(10F+4E^2)(-\gamma)^3\Big) J(\gamma)\\
&&+ (-\gamma) \Big(E(-\gamma)+2(E^2+2F)\gamma^2+7EF(-\gamma)^3\Big)\det H(\gamma) \\
&&+\Big( \gamma^2+(-\gamma)^3 2E\Big)  {\cal L}_{H(\gamma)} (D(E, \gamma))+2  (-\gamma) ^3 {\cal L}_{H(\gamma)} (D(F, \gamma))-(\gamma)^3{\cal H(\gamma}(H(E))
 \\
&&-(1+2E(-\gamma)) \gamma^2 {\cal L}_{H(E)}(\d \gamma\otimes \dbar \gamma)-(-\gamma)^3 {\cal L}_{H(F)}(\d \gamma \otimes \dbar \gamma) +O(\gamma^4 ).
\end{eqnarray*}
Therefore,
\begin{eqnarray}\label{J(j)/J}
{J(g)\over J} 
&=&1  +(-\gamma)(4E) +I_2(-\gamma)^2+I_3(-\gamma)^3+O(\gamma^4),
\end{eqnarray}
where
$$
I_2={1\over J} \Big((3F+6E^2) J+E \det H(\gamma) +{\cal L}_{H(\gamma)}(D(E, \gamma)-{\cal L}_{\d \gamma \otimes \dbar \gamma}(H(E))\Big)
$$
and
\begin{eqnarray*}
I_3&=&{1\over J} \Big[E(10F+4E^2)J +(4F+2E^2)\det H(\gamma)+2E{\cal L}_{H(\gamma)}(D(E, \gamma)) \\
&&+  2  {\cal L}_{H(\gamma)} ( D(F, \gamma) ) ) -{\cal L}_{H(\gamma)}(H(E)) -2E {\cal L }_{\d \gamma\otimes \dbar \gamma}(H(E)) 
-{\cal L}_{\d \gamma \otimes \dbar \gamma}(H(F)) \Big]+O(\gamma^4 ).
\end{eqnarray*}

Since
\begin{eqnarray*}
\lefteqn{\d g\otimes \dbar g- gH(g)}\\
&=&a^2  \d \gamma \otimes \dbar \gamma -(\gamma)^2  D(E,\gamma) -(-\gamma)^3 ( 2E D(E, \gamma)+ D(F, \gamma) )+\gamma^4 A_1\\
&&- g\Big(a H(\gamma)-2\Big( E+3F(-\gamma) \Big)  \d \gamma \otimes \dbar \gamma) +2 (-\gamma) D(E, \gamma)\Big)
\\
&&-g \Big(3 \gamma^2 D(F, \gamma)
-  \gamma^2 H(E) -(-\gamma)^3 H(F)\Big)\\
&=&a^2  \d \gamma \otimes \dbar \gamma -(\gamma)^2  D(E,\gamma) +O(\gamma^3)\\
&&- g\Big(a H(\gamma)-2E \d \gamma \otimes \dbar \gamma -6F(-\gamma)  \d \gamma \otimes \dbar \gamma) +2 (-\gamma) D(E, \gamma)\Big).
\end{eqnarray*}
We get that
\begin{eqnarray*}
\lefteqn{\d g\otimes \dbar g- gH(g)}\\
&=&(a^2 +2E g  +6F g(-\gamma)) \d \gamma \otimes \dbar \gamma -(2g (-\gamma)+(\gamma)^2 ) D(E,\gamma)-a g H(\gamma)+ O(\gamma^3) \\
&=&(1+4E(-\gamma)+4E^2(-\gamma)^2 -2E (-\gamma)-2E^2(-\gamma)^2  ) \d \gamma \otimes \dbar \gamma \\
&&+(\gamma)^2  D(E,\gamma)+(-\gamma)(1+3E(-\gamma)) H(\gamma)+ O(\gamma^3) \\
&=&(1+2E(-\gamma)+2E^2(-\gamma)^2  ) \d \gamma \otimes \dbar \gamma \\
&&+(\gamma)^2  D(E,\gamma)+(-\gamma)(1+3E(-\gamma)) H(\gamma)+ O(\gamma^3). 
\end{eqnarray*}
Therefore,
\begin{eqnarray*}
\lefteqn{{\cal L}_{H(J)} (\d g\otimes \dbar g- gH(g))}\\
&=& (1+2E(-\gamma)+2E^2(-\gamma)^2  ) {\cal L}_{H(J)} (\d \gamma \otimes \dbar \gamma) \\
&&+(-\gamma)^2  {\cal L}_{H(J)} (D(E,\gamma))+(-\gamma)(1+3E(-\gamma)) {\cal  L}_{H(J)}(H(\gamma))+ O(\gamma^3).
\end{eqnarray*}
Thus,
\begin{eqnarray*}
\lefteqn{(-g) {\cal L}_{H(J)} (\d g\otimes \dbar g- gH(g))}\\
&=&(-\gamma) (1+3E(-\gamma)+(F+4E^2)(-\gamma)^2  ) {\cal L}_{H(J)} (\d \gamma \otimes \dbar \gamma) \\
&&+(-\gamma)^3 {\cal L}_{H(J)}(D(E,\gamma))+(-\gamma)^2(1+4E(-\gamma)) {\cal  L}_{H(J)} (H(\gamma))+ O(\gamma^4).
\end{eqnarray*}
Moreover,
\begin{eqnarray*}
\lefteqn{{\cal L}_{\d J\otimes \dbar J} (\d g\otimes \dbar g- gH(g))}\\
&=&(1+2E(-\gamma)+(3F+2E^2)(-\gamma)^2  ) {\cal L}_{\d J\otimes \dbar J} (\d \gamma \otimes \dbar \gamma) \\
&&+(\gamma)^2  {\cal L}_{\d J\otimes \dbar J}(D(E,\gamma))+(-\gamma)(1+3E(-\gamma)) {\cal  L}_{\d J\otimes \dbar J} (H(\gamma))+ O(\gamma^3).
\end{eqnarray*}
Thus,
\begin{eqnarray*}
\lefteqn{ g {\cal L}_{\d J\otimes \dbar J} (\d g\otimes \dbar g- gH(g))}\\
&=&-(-\gamma) (1+3E(-\gamma)+(4F+4E^2)(-\gamma)^2  ) {\cal L}_{\d J\otimes \dbar J} (\d \gamma \otimes \dbar \gamma) \\
&&-(-\gamma)^3  {\cal L}_{\d J\otimes \dbar J}(D(E,\gamma))- (-\gamma)^2(1+4E(-\gamma)) {\cal  L}_{\d J\otimes \dbar J} (H(\gamma))+ O(\gamma^4).
\end{eqnarray*}
Therefore
\begin{eqnarray*}
\lefteqn{(-g) {\cal L}_{H(J)} (\d g\otimes \dbar g- gH(g))+  {g \over J} {\cal L}_{\d J\otimes \dbar J} (\d g\otimes \dbar g- gH(g)) }\\
&=&(-\gamma) (1+3E(-\gamma)+(4F+4E^2)(-\gamma)^2  ) {\cal L}_{H(J)} (\d \gamma \otimes \dbar \gamma) \\
&&+(-\gamma)^3 {\cal L}_{H(J)} D(E,\gamma)+(-\gamma)^2(1+4E(-\gamma)) {\cal  L}_{H(J)} (H(\gamma))\\
&&-{(-\gamma) \over J}(1+3E(-\gamma)+(4F+4E^2)(-\gamma)^2  ) {\cal L}_{\d J\otimes \dbar J} (\d \gamma \otimes \dbar \gamma) \\
&&-{(-\gamma)^3\over J}   D(E,\gamma)- {(-\gamma)^2\over J} (1+4E(-\gamma)) {\cal  L}_{\d J\otimes \dbar J} (H(\gamma))+ O(\gamma^4) \\
&=&G_1(-\gamma)+G_2(-\gamma)^2+G_3(-\gamma)^3+O(\gamma^4),
\end{eqnarray*}
where
$$
G_1= \Big({\cal L}_{H(J)}(\d \gamma \otimes \dbar \gamma)-{1\over J} {\cal  L}_{\d J\otimes \dbar J}(\d \gamma\otimes \dbar \gamma)\Big)
=J{\cal L}_{H(\log J)}(\d \gamma \otimes \dbar \gamma),
$$
$$
G_2=\Big(3E  G_1 +{\cal L}_{H(J)}(H(\gamma))-{1\over J}{\cal L}_{\d J\otimes \dbar J} (H(\gamma))\Big)
=3EG_1+J {\cal L}_{H(\log J)}(H(\gamma))
$$
and
\begin{eqnarray*}
G_3&=& (F+4E^2)G_1+{\cal L}_{H(J)} (D(E, \gamma) )-{1\over J} {\cal L}_{\d J\otimes \dbar J} (D(E,\gamma))\\
&&+4E {\cal L}_{H(J)}(H(\gamma))
-{4E \over J} {\cal L}_{\d J\otimes \dbar J}(H(\gamma))\\
&=&4(F+E^2) G_1+J {\cal L}_{H(\log J)}(D(E, \gamma))+4JE {\cal L}_{H(\log J)}(H(\gamma)).
\end{eqnarray*}
Therefore,
we have
\begin{eqnarray}
J(fg)&=&1+(4E+{G_1\over 3J^2})(-\gamma)+(I_2+{G_2\over 3 J^2})
(-\gamma)^2\nonumber\\
&&+\Big(I_3+{G_3\over 3 J^2}+{1 \over 9J}\det (H(\log J)\Big)(-\gamma)^3+O(\gamma^4).
\end{eqnarray}
Now we let
\begin{eqnarray}
E=-{G_1\over 12 J^2}=-{1\over 12 J}  \Big({\cal L}_{H(\log J)}(\d \gamma \otimes \dbar \gamma)\Big).
\end{eqnarray}
Then
$$
4E+{G_1\over 3 J^2}=0.
$$
Since
\begin{eqnarray*}
(I_2+{G_2\over 3 J^2})&=&3F+6E^2+{ \det H(\gamma)\over J} E+
{1\over J}{\cal L}_{H(\gamma)}(D(E, \gamma))-{1\over J}{\cal L}_{\d \gamma \otimes \dbar \gamma}(H(E))\\
&&+{E\over J^2} G_1 +{1\over 3 J}{\cal L}_{H(\log J)}(H(\gamma))\\
&=&3F+6E^2+{\det H(\gamma)\over J} E -12 E^2\\
&&+{1\over J} \Big({\cal L}_{H(\gamma)}(D(E, \gamma))-{\cal L}_{\d \gamma \otimes \dbar \gamma}(H(E))
+{1\over 3 }{\cal L}_{H(\log J)}(H(\gamma))\Big)\\
&=& 0,
\end{eqnarray*}
if we let
\begin{eqnarray*}
F&=&2E^2- {1 \over 3 J} \Big[ E  \det H(\gamma)+
{\cal L}_{H(\gamma)}(D(E, \gamma))-{\cal L}_{\d \gamma \otimes \dbar \gamma}(H(E))
+{1\over 3 }{\cal L}_{H(\log J)}(H(\gamma))\Big],
\end{eqnarray*}
we obtain
$$
E(10F+4E^2)+(F+4E^2){G_1\over 3J^2}=E(10F+4E^2)+(F+4E^2)(-4E)=E(6F-12 E^2).
$$
We now introduce
\begin{eqnarray*}
{\cal A}:&=& I_3+{G_3\over 3J^2}+{1\over 9J} \det H(\log J)\\
&=&6 E(F-E^2) +(4F+2E^2) {\det H(\gamma)\over J} +{2E \over J} {\cal L}_{H(\gamma)}(D(E, \gamma))\\
&&+ {1\over J} \Big[ 2 {\cal L}_{H(\gamma)} ( D(F, \gamma) ) ) -2E {\cal L}_{\d \gamma \otimes \dbar \gamma}(H(E)-{\cal L }_{H(\gamma)}(H(E)) 
-{\cal L}_{\d \gamma \otimes \dbar \gamma}(H(F)) \Big]\\
&&+{1\over 3J}  {\cal L}_{H(\log J)}(D(E, \gamma))+{4E \over 3 J}  {\cal L}_{H(\log J)}(H(\gamma))
+{1\over 9J } \det H(\log J).
\end{eqnarray*}

\subsection{Preliminaries for the proof of Theorem 1.1}
In this section, we apply the above outcome for the case of real ellipsoids in $\CC^{2}$. Note that, after complex linearly changes of variables, any
real ellipsoid   in $\mathbb{C}^2$ can be written by
\begin{equation}
E(A)=\{z\in \CC^2: \gamma(z)<0\},
\end{equation}
where
\begin{equation}
\gamma (z)=|z|^2+\sum_{j=1}^2A_j \hbox{Re}( z_j^2)-1,\quad A_j \in [0, 1).
\end{equation}
Then we compute:
$$
\gamma_j=\zbar_j+A_j z_j,
$$
$$
J(\gamma)=-\gamma+ |\d \gamma|^2 =1+\sum_{j=1}^2A_j^2 |z_j|^2+\sum_{j=1}^2 A_j \hbox{Re}  z_j^2.
$$
Therefore 
$$
J_j=A_j^2 \zbar_j+A_j z_j=A_j (z_j+A_j \zbar_j)=A_j \gamma_{\jbar},
$$
and
$$
H(\log J)={1\over J} (D(A^2)-{1\over J} A\dbar \gamma \otimes A\d \gamma),
$$
where
$D(A^2)$ is the diagonal matrix with diagonal entries are $A_1^2$ and $A_2^2$ and $A\dbar \gamma=(A_1\gamma_{\bar1}, A_2 \gamma_{\bar2})$. 
\medskip

\subsection {The function $E$ and its derivatives}
Let us first compute
\begin{eqnarray*}
E&=&-{1\over 12 J} {\cal L}_{H(\log J}(\d \gamma\otimes \dbar \gamma)\nonumber\\
&=&-{1\over 12 J^3} (A_2^2 |\gamma_1|^2+A_1^2 |\gamma_2|^2)J(\gamma)+{1\over 12 J^3}|A_1 \gamma_{\bar1} \gamma_2-A_2\gamma_{\bar2} \gamma_1|^2.\\
\end{eqnarray*}
Now, with $J(\gamma)=-\gamma+|\d \gamma|^2$, one has
\begin{eqnarray*}
f(z)&=& (A_2^2|\gamma_1|^2+A_1^2|\gamma_2|^2) J-|A_1 \gamma_{\bar1} \gamma_2-A_2\gamma_{\bar2} \gamma_1|^2\\
&=& A_1^2|\gamma_2|^4+A_2^2 |\gamma_1|^4+(-\gamma)(A_1^2|\gamma_2|^2+A_2^2|\gamma_1|^2) +A_1A_2  2\hbox{Re}(\gamma_{\bar1}^2\gamma_2^2).
\end{eqnarray*}
Then
\begin{eqnarray}
E=-{f\over 12 J^3}.
\end{eqnarray}
Moreover,
\begin{eqnarray*}
\d_1 f&=&2A_2^2 ( \gamma_{\bar1} \gamma_1^2 +A_1 \gamma_1 \gamma_{\bar1}^2)+(-\gamma) A_2^2 (\gamma_1+A_1 \gamma_{\bar1})\\
&&-\gamma_1(A_1^2|\gamma_2|^2
+A_2^2|\gamma_1|^2)
+2A_1A_2 (\gamma_2^2 \gamma_{\bar1} +A_1 \gamma_{\bar2}^2 \gamma_1)\\
&=&A_2^2 ( \gamma_{\bar1} \gamma_1^2 +2A_1 \gamma_1 \gamma_{\bar1}^2)+(-\gamma) A_2^2 (\gamma_1+A_1 \gamma_{\bar1})\\
&&- A_1^2 \gamma_1|\gamma_2|^2
+2A_1A_2 (\gamma_2^2 \gamma_{\bar1} +A_1 \gamma_{\bar2}^2 \gamma_1),
\end{eqnarray*}
\begin{eqnarray*}
\d_2 f&=&A_1^2  (2 A_2 \gamma_2 \gamma_{\bar2}^2 +2 \gamma_2^2 \gamma_{\bar2}) +(-\gamma) A_1^2  (\gamma_2+A_2 \gamma_{\bar2}) \\
&&-\gamma_2 (A_1^2|\gamma_2|^2+A_2^2|\gamma_1|^2)
+2A_1A_2 ( A_2 \gamma_2 \gamma_{\bar1} ^2+ \gamma_{\bar2} \gamma_1^2)\\
&=&A_1^2  (2 A_2 \gamma_2 \gamma_{\bar2}^2 + \gamma_2^2 \gamma_{\bar2}) +(-\gamma) A_1^2  (\gamma_2+A_2 \gamma_{\bar2}) \\
&&-A_2^2 \gamma_2 |\gamma_1|^2
+2A_1A_2 ( A_2 \gamma_2 \gamma_{\bar1} ^2+ \gamma_{\bar2} \gamma_1^2),
\end{eqnarray*}
\begin{eqnarray*}
\d_{1\bar1}  f&=&A_2^2 ( A_1 \gamma_1^2 +2 | \gamma_1|^2 +2A_1 \gamma_{\bar1}^2 +4A_1^2|\gamma_1|^2 )
- A_2^2 (|\gamma_1|^2 +A_1 \gamma_{\bar1}^2)+(-\gamma) A_2^2 (1+A_1^2)\\
&&-(A_1^2|\gamma_2|^2
+A_2^2|\gamma_1|^2)- A_2^2 (|\gamma_1|^2 +A_1 \gamma_1 ^2)+2A_1^2A_2 ( \gamma_2^2 + \gamma_{\bar2}^2 )\\
&=&A_2^2 ( A_1 \gamma_1^2 + |\gamma_1|^2 +A_1 \gamma_{\bar1}^2 +4A_1^2|\gamma_1|^2 )
+(-\gamma) A_2^2 (1+A_1^2)\\
&&-A_1^2|\gamma_2|^2+2A_1^2A_2 ( \gamma_2^2 + \gamma_{\bar2}^2 )\\
\end{eqnarray*}
and
\begin{eqnarray*}
\d_{2\bar2} f&=&A_1^2 \d_{2\bar2} |\gamma_2|^4-A_1 \gamma_2  \d_2 |\gamma_2|^2+
 (-\gamma) A_1 \d_{2\bar 2}  |\gamma_2|^2\\
 &&- (A_1^2|\gamma_2|^2+A_2^2|\gamma_1|^2)- A_1^2 \gamma_2 \d_{\bar2} |\gamma_2|^2)
+2A_1A_2^2 (   \gamma_{\bar1} ^2+  \gamma_1^2)\\
&=&2A_1^2(A_2\gamma_2^2+A_2 \gamma_{\bar2}^2 +2(1+A_2^2) |\gamma_2|^2)
-A_1^2(|\gamma_2|^2+A_2 \gamma_{\bar2}^2)+(-\gamma)A_1^2(1+A_2^2)\\
&&-(A_1^2 |\gamma_2|^2+A_2^2|\gamma_1|^2)-A_1^2(|\gamma_2|^2+A_2 \gamma_2^2)
+2A_1A_2^2(\gamma_1^2+\gamma_{\bar1}^2)\\
&=&A_1^2(A_2\gamma_2^2+A_2 \gamma_{\bar2}^2 +(1+4A_2^2) |\gamma_2|^2)
+(-\gamma)A_1^2(1+A_2^2)\\
&&-A_2^2|\gamma_1|^2
+2A_1A_2^2(\gamma_1^2+\gamma_{\bar1}^2).
\end{eqnarray*}

Let $z_0\in \d D(A)$ with $\gamma_1(z_0)=0$. Then
$$
E(z_0)=-{A_1^2\over 12 J}, \quad f(z_0)=A_1^2 J^2,
$$
$$
\d_j E=-{\d_j f\over 12 J^3} +{f \over 4 J^4} \d_j J\quad \hbox{and}\quad  \d_1 f(z_0)=0.
$$
 This implies that
$$
\d_1 E(z_0)=0\quad\hbox{and}\quad  \d_2 E(z_0)=-{A_1^2 (A_2 \gamma_{\bar2}+\gamma_2)
\over 12 J^2}+{A_2 A_1^2\over 4 J^2} \gamma_{\bar2}={A_1^2 (A_2 \gamma_{\bar2}-\gamma_2)\over 12 J^2}.
$$
Then
$$
2\hbox{Re}(\gamma_{\bar2} E_2(z_0))={A_1^2 A_2 2\hbox{Re} \gamma_2^2 \over 12J^2}-{A_1^2\over 6J} \quad
\hbox{ and}\quad \gamma_2 E_2(z_0)={A_1^2\over 12 J}\Big(A_2-{\gamma_2^2 \over  J} \Big).
$$
Notice that
$$
\d_{j\jbar } E=-{\d_{j\jbar} f\over 12 J^3}+{f _j J_{\bar j}  +f_{\bar j} J_j \over 4 J^4} -{f \over J^5} J_{\bar j} J_j +{f\over 4 J^4} A_j^2.
$$
Therefore,
\begin{eqnarray*}
E_{1\bar1}(z_0)&=&-{f_{1\bar1} (z_0) \over 12 J^3}+{A_1^2 J^2\over 4J^4} A_1^2\\
&=&-{-A_1^2 J+2A_1^2A_2(\gamma_2^2+\gamma_{\bar2}^2) \over 12 J^3}+{A_1^4\over 4J^2}\\
&=& -{A_1^2 A_2(\gamma_2^2+\gamma_{\bar2}^2)\over 6 J^3}+{A_1^2 +3A_1^4\over 12 J^2}
\end{eqnarray*}
and 
\begin{eqnarray*}
E_{2\bar2}(z_0)&=&-{ f_{2\bar 2} \over 12 J^3}+{f _2 J_{\bar 2}  +f_{\bar 2} J_2 \over 4 J^4 } -{f \over J^5} J_{\bar 2} J_2+{f\over 4 J^4} A_2^2\\
&=&-{ A_1^2(A_2\gamma_2^2+A_2 \gamma_{\bar2}^2 +(1+4A_2^2) |\gamma_2|^2) \over 12 J^3}\\
&&+A_2 A_1^2 {(4A_2|\gamma_2|^4+|\gamma_2|^2\gamma_2^2 +
|\gamma_2|^2 \gamma_{\bar2}^2)  \over 4 J^4 } -{A_1^2 J^2 A_2^2 |\gamma_2|^2  \over J^5} +{A_1^2 A_2^2\over 4 J^2} \\
&=&-{ A_1^2 A_2 (\gamma_2^2+ \gamma_{\bar2}^2 ) \over 12 J^3}
-{ A_1^2 (1+4A_2^2)  \over 12 J^2}+4{A_1^2 A_2^2  \over 4 J^2 } 
 +A_2 A_1^2 {\gamma_2^2  +
 \gamma_{\bar2}^2  \over 4 J^3 } 
-{3A_1^2  A_2^2 \over 4 J^2} \\
&=&{ A_1^2 A_2 (\gamma_2^2+ \gamma_{\bar2}^2 ) \over 6 J^3}
+{ -4A_1^4 +3A_1^2A_2^2 -A_1^2  \over 12 J^2}.
\end{eqnarray*}
Since
\begin{eqnarray*}
 f_{22}(z_0)&=&\d_2 \Big(A_1^2  (2 A_2 \gamma_2 \gamma_{\bar2}^2 + \gamma_2^2 \gamma_{\bar2}) +(-\gamma) A_1^2  (\gamma_2+A_2 \gamma_{\bar2}) \\
&&\qquad -A_2^2 \gamma_2 |\gamma_1|^2
+2A_1A_2 ( A_2 \gamma_2 \gamma_{\bar1} ^2+ \gamma_{\bar2} \gamma_1^2)\Big)\\
&=&A_1^2(2A_2^2 \gamma_{\bar2}^2+6A_2|\gamma_2|^2+\gamma_2^2)- A_1^2(\gamma_2^2+A_2|\gamma_2|^2)\\
&=&A_1^2(2A_2^2 \gamma_{\bar2}^2+5A_2|\gamma_2|^2),
\end{eqnarray*}
one has
\begin{eqnarray*}
E_{2 2}(z_0)&=&-{f_{22}\over 12 J^3}+{f_2 J_2\over 2 J^4}+{A_2f \over 4J^4}-{f(J_2)^2\over J^5}\\
&=&-{A_1^2 A_2^2 \gamma_{\bar2}^2 \over 6 J^3}-{5 A_1^2A_2 \over 12 J^2}+A_2 A_1^2 {(2 A_2 |\gamma_2|^2 \gamma_{\bar2}^2+|\gamma_2|^4)\over 2J^4}
+A_2{A_1^2|\gamma_2|^4 \over 4J^4}-A_2^2 {A_1^2 |\gamma_2 |^4 \gamma_{\bar 2}^2 \over J^5}\\
&=&-{A_1^2 A_2^2 \gamma_{\bar2}^2 \over 6 J^3}+{A_1^2A_2 \over 3 J^2}.
\end{eqnarray*}
Since
$$
f_{11}(z_0)=-A_1^3 |\gamma_2|^2+2A_1A_2 (\gamma_2^2+A_1^2 \gamma_{\bar2}^2)
$$
and
\begin{eqnarray*}
\d_{1\bar1}  f
&=&A_2^2 ( A_1 \gamma_1^2 + |\gamma_1|^2 +A_1 \gamma_{\bar1}^2 +4A_1^2|\gamma_1|^2 )
+(-\gamma) A_2^2 (1+A_1^2)\\
&&-A_1^2|\gamma_2|^2+2A_1^2A_2 ( \gamma_2^2 + \gamma_{\bar2}^2 ),
\end{eqnarray*}
one has
\begin{eqnarray*}
\d_{2 1\bar1}  f(z_0)
&=&-\gamma_2 A_2^2 (1+A_1^2)-A_1^2 (\gamma_2+A_2 \gamma_{\bar2} )+2A_1^2A_2 (2 A_2 \gamma_2+2 \gamma_{\bar2})\\
&=&\gamma_2 (3A_1^2A_2^2 -A_1^2-A_2^2) +3A_1^2A_2  \gamma_{\bar2}\\
\end{eqnarray*}
and
\begin{eqnarray*}
\d_{11\bar1}  f (z_0)
&=&0.
\end{eqnarray*}
Notice that
$$
\d_{1\bar1} |\gamma_1|^2=1+A_1^2,\quad \d_{1\bar1} \gamma_1^2=2A_1,
$$
one has
\begin{eqnarray*}
\d_{1\bar1 1\bar1}  f
&=&A_2^2 ( A_12 A_1+ (1+A_1^2) +A_1 2A_1 +4A_1^2(1+A_1^2) )
-A_2^2 (1+A_1^2)\\
&=&4A_1^2A_2^2(2+A_1^2).
\end{eqnarray*}
Since
\begin{eqnarray*}
\d_{  1 \1bar } E&=&-{f_{1\bar1} \over 12 J^3}+{f _1 J_{\bar 1}  +f_{\bar 1} J_1\over 4 J^4} -{f \over J^5} J_{\bar 1} J_1+{f\over 4 J^4} A_1^2,
\end{eqnarray*} 
one has
\begin{eqnarray*}
\d_{2 1 \1bar } E(z_0)&=&-{f_{21\bar1} \over 12 J^3}+{f_{1\bar1} J_2 \over 4 J^4} +{f_2\over 4 J^4} A_1^2-{fJ_2 \over J^5}\\
&=&-{(3A_1^4A_2^2-A_1^2-A_2^2)\gamma_2+3A_1^2 A_2 \gamma_{\bar2} \over 12 J^3}\\
&&+A_2 \gamma_{\bar2} {-A_1^2|\gamma_2|^2+2A_1^2A_2(\gamma_2^2+\gamma_{\bar2}^2)  \over 4 J^4} +{A_1^2 (2A_2 \gamma_{\bar2}+\gamma_2) \over 4 J^3} A_1^2-{A_1^2 A_2 \gamma_{\bar2}  \over J^3}\\
&=&{(3A_1^2A_2^2+A_1^2+A_2^2+3A_1^4)\gamma_2 \over 12 J^3} + A_1^2A_2^2{ \gamma_{\bar2}^3  \over 2 J^4} +{ A_1^4 A_2 \gamma_{\bar2} \over 2 J^3} 
-3 {A_1^2 A_2 \gamma_{\bar2}  \over 2 J^3}.
\end{eqnarray*} 
Thus,
\begin{eqnarray*}
\lefteqn{2\hbox{Re} \gamma_{\bar2} \d_{2 1 \1bar } E(z_0)}\\
&=&{(3A_1^2A_2^2+A_1^2+A_2^2+3A_1^4) \over 6 J^2} + A_1^2A_2^2{ 2\hbox{Re} \gamma_{\bar2}^4  \over 2 J^4} +{ A_1^4 A_22\hbox{Re} \gamma_{\bar2} ^2\over 2 J^3} 
-3 {A_1^2 A_2 2\hbox{Re} \gamma_{\bar2}^2  \over 2 J^3}.
\end{eqnarray*} 
Since
\begin{eqnarray*}
\d_{\bar1} f
&=&A_2^2 ( 2 A_1 \gamma_{\bar1} \gamma_1^2 +\gamma_1  \gamma_{\bar1}^2)+(-\gamma) A_2^2 (\gamma_{\bar1}+A_1 \gamma_1 )\\
&&- A_1^2 \gamma_{\bar1} |\gamma_2|^2
+2A_1A_2 (A_1 \gamma_2^2 \gamma_{\bar1} + \gamma_{\bar2}^2 \gamma_1),
\end{eqnarray*}
and
\begin{eqnarray*}
\d_{\bar1 \bar1} f
&=&A_2^2 ( 2 A_1^2 \gamma_1^2 +4A_1 |\gamma_1|^2 +\gamma_{\bar1}^2 +2A_1 |\gamma_1 |^2 )
+(-\gamma) A_2^2 2A_1-\gamma_{\bar1} A_2^2(\gamma_{\bar1}+A_1 \gamma_1)\\
&&- A_1^3 |\gamma_2|^2
+2A_1A_2 (A_1^2 \gamma_2^2 + \gamma_{\bar2}^2 ).
\end{eqnarray*}
Then 
$$
f_{\bar1\bar1}(z_0)=-A_1^3|\gamma_2|^2+2A_1A_2(A_1^2 \gamma_2^2+\gamma_{\bar2}^2),
$$
\begin{eqnarray*}
f_{2\bar1\bar1}(z_0)
&=&-2A_1A_2^2 \gamma_2 
- A_1^3 (\gamma_2+A_2\gamma_{\bar2} )
+4A_1A_2 (A_1^2 A_2 \gamma_2 +  \gamma_{\bar2} ),
\end{eqnarray*}
and
\begin{eqnarray*}
E_{2\bar1\bar1}(z_0)
&=&-{f_{2\bar1\bar1}\over 12 J^3}+{f_{\bar1\bar1} J_2\over 4 J^4}+{f_2\over 4J^4} A_1-A_1{f\over J^5} J_2\\
&=&-{-2A_1 A_2^2 \gamma_2 
- A_1^3 ( \gamma_2+A_2\gamma_{\bar2} )
+4A_1A_2 (A_1^3 \gamma_2 + \gamma_{\bar2} ) \over 12 J^3}\\
&&+A_2 \gamma_{\bar2} {-A_1^3|\gamma_2|^2+2A_1A_2(A_1^2 \gamma_2^2+\gamma_{\bar2}^2) \over 4 J^4}+{A_1^2(2A_2\gamma_{\bar2}+\gamma_2)\over 4J^3} A_1-A_1A_2 {A_1^2 \over J^3} \gamma_{\bar2}\\
&=&{(2A_1 A_2^2 
+ A_1^3 -A_1^4A_2)  \gamma_2+(A_1^3 A_2- 4A_1 A_2) \gamma_{\bar2}  \over 12 J^3}\\
&&+  {-3A_1^3 A_2   \gamma_{\bar2} +(2A_1^3 A_2^2  +A_1^3) \gamma_2 \over 4J^3} +{2A_1A_2^2\over 4J^4}  \gamma_{\bar2}^3\\
&=&{(2A_1 A_2^2 
+ 4A_1^3+6A_1^3A_2^2 -A_1^4A_2)  \gamma_2+(-8A_1^3 A_2- 4A_1 A_2) \gamma_{\bar2}  \over 12 J^3}+{2A_1A_2^2\over 4J^4}  \gamma_{\bar2}^3.
\end{eqnarray*}
Thus,
\begin{eqnarray*}
2\hbox{Re} \gamma_{\bar2} E_{2\bar1\bar1}(z_0)
&=&{(2A_1 A_2^2 
+ 4A_1^3+6A_1^3A_2^2 -A_1^4A_2)  \over 6 J^2}\\
&&+{(-8A_1^3 A_2- 4A_1 A_2)   \over 12 J^3} 2\hbox{Re}(\gamma_2^2) +{2A_1A_2^2\over 4J^4} 2\hbox{Re}( \gamma_{\bar2}^4).
\end{eqnarray*}
Notice that
$$
J_{11}=A_1,\  J_{1\bar1}=A_1^2,\ f_{11}(z_0)=-A_1^2 J+2A_1A_2(A_1^2 \gamma_{\bar2}^2+\gamma_2^2),
$$
$$
\ f_{1\bar1}(z_0)=-A_1^2 J+2A_1A_2(\gamma_2^2+\gamma_{\bar2}^2)
$$
and
$$
f_{1\bar1 1\bar1}(z_0)=4A_1^2A_2^2(2+A_1^2),
$$
one has
\begin{eqnarray*}
\lefteqn{4f_{1\bar1} J_{1\bar1 }+ f_{11} J_{\bar1\bar1}+f_{\bar1\bar1} J_{11} }\\
&=&-4A_1^4 J+8A_1^3A_2(\gamma_2^2+\gamma_{\bar2}^2)
+A_1 2\hbox{Re}(-A_1^2 J+2A_1 A_2(A_1^2 \gamma_2^2+\gamma_{\bar2}^2)\\
&=&(-2A_1^3-4A_1^4) J +(8A_1^3A_2 +2A_1^2A_2(1+A_1^2)) 2\hbox{Re}(\gamma_2^2).
\end{eqnarray*}

\begin{eqnarray*}
\d_{1\bar1  1 \1bar } E(z_0)&=&-{f_{1 \bar1 1\bar1} \over 12 J^3}
+{f_{1\bar1} \over 4 J^4}J_{1\bar1} +{2 f_{1\bar1} J_{1\bar1}+f_{11} J_{\bar1 \bar1}+f_{\bar1\bar1} J_{11} \over 4 J^4} \\
&& -{f \over J^5} A_1^2(1+A_1^2)+{f_{1\bar1} \over 4 J^4} J_{1\bar1}- {f\over  J^5} A_1^4\\
&=&-{4A_1^2 A_2^2(2+A_1^2) \over 12 J^3}-{A_1^2  \over J^3} A_1^2(1+2A_1^2)\\
&&+{(-2A_1^3-4A_1^4) J \over 4J^4}+{(8A_1^3A_2 +2A_1^2A_2(1+A_1^2)) \over 4J^4} 2\hbox{Re}(\gamma_2^2)\\
\\
&=&-{6A_1^6 +2A_1^4 A_2^2 +12A_1^4+3A_1^3+4A_1^2A_2^2 \over 6 J^3}\\
&&+{(8A_1^3A_2 +2A_1^2A_2(1+A_1^2)) \over 4J^4} 2\hbox{Re}(\gamma_2^2).
\end{eqnarray*}

\subsection{The function $F$ and its derivatives}
 Notice that 
 \begin{eqnarray*}
F&=&2E^2- {1 \over 3 J} \Big( E +
{\cal L}_{H(\gamma)}(D(E, \gamma))-{\cal L}_{\d \gamma \otimes \dbar \gamma}(H(E))
+{1\over 3 }{\cal L}_{H(\log J)}(H(\gamma))\Big)\\
&=&2 E^2-{E\over 3 J} -{1\over 3J} 2\hbox{Re}( \gamma_{\bar1} E_1+\gamma_{\bar2} E_2)
+{1\over 3J}\Big(|\gamma_2|^2 E_{1\bar1}+|\gamma_1|^2 E_{2\bar2}-2\hbox{Re} (\gamma_1 \gamma_{\bar2} E_{2\bar1})\Big)\\
&&-{A_1^2 +A_2^2 \over 9 J^2} +{A_1^2 |\gamma_1|^2 +A_2^2 |\gamma_2|^2\over 9J^3}
\end{eqnarray*}
Since $E_2(z_0)=A_1^2 A_2 \gamma_{\bar2}(z_0) /6J^2$, one has
$$
2\hbox{Re}(\sum_{j=1}^2 \gamma_{\bar j}(z_0) E_j(z_0) )=A_1^2 A_2 {\gamma_2(z_0)^2+\gamma_{\bar2} (z_0)^2\over 6J^2}
$$
and 
$$
E_{1\bar1}(z_0)=-{A_1^2 A_2(\gamma_2^2+\gamma_{\bar2}^2)\over 6 J^3}+{A_1^2 +3A_1^4\over 12 J^2}.
$$
 Then
\begin{eqnarray*}
F(z_0)&=&2E(z_0)^2-{E(z_0)\over 3J}-{2\hbox{Re}(\gamma_{\bar2} E_2)\over 3J}+{E_{1\bar1}(z_0)\over 3}-{A_1^2 \over 9 J^2}\\
&=& 2{A_1^4 \over (12)^2J^2} +{A_1^2 \over 36 J^2} -A_1^2A_2 {\gamma_2^2+\gamma_{\bar2}^2\over 36J^3} +{A_1^2\over 18J}
-{A_1^2 A_2(\gamma_2^2+\gamma_{\bar2}^2)\over 18 J^3}+{A_1^2 +3A_1^4\over 36 J^2}-{A_1^2\over 9J^2}\\
&=& {7 A_1^4 \over 72 J^2}  -A_1^2A_2 {\gamma_2^2+\gamma_{\bar2}^2\over 12 J^3} .
\end{eqnarray*}

Notice that with $\gamma_1(z_0)=0$ and $E(z_0)=-A_1^2/(12J)$, one has
\begin{eqnarray*}
\d_2 F(z_0)&=&4E E_2(z)-{E_2\over 3J}+{E A_2 \gamma_{\bar2}\over 3J^2}-{1\over 3J} (E_2+\gamma_{\bar2} E_{22}+A_2 E_{\bar2}+\gamma_{2} E_{2\bar2})\\
&&+{1\over 3J} (|\gamma_2|^2 E_{1\bar1 2}+(\gamma_2+A_2 \gamma_{\bar2})E_{1\1bar})- {A_2 \gamma_{\bar2} \over 3J} E_{1\bar1}\\
&&+{2\over 9} {A_1^2+A_2^2\over J^3} A_2 \gamma_{\bar2}
+{A_2^2(\gamma_2+A_2 \gamma_{\bar2} )\over 9J^3} -{A_2^2  \over 3 J^3} A_2 \gamma_{\bar2}\\
&=&-{(A_1^2+2) E_2\over 3 J}+7 {A_1^2  A_2 \gamma_{\bar2}\over 36 J^3}-{1\over 3J} (\gamma_{\bar2} E_{22}+A_2 E_{\bar2}+\gamma_{2} E_{2\bar2})\\
&&+{1\over 3} E_{1\bar1 2}+{1\over 3J} \gamma_2E_{1\1bar}
+{A_2^2 \gamma_2\over 9J^3}
\end{eqnarray*}
and
$$
2\hbox{Re}(\gamma_{\bar2} E_2)={A_1^2(A_2 2\hbox{Re}(\gamma_2^2)-2J \over 12 J^2}\  \hbox{ and } \ 2\hbox{Re}(\gamma_{\bar2} E_{\bar2})={A_1^2\over 12 J^2} (2J A_2-2\hbox{Re}(\gamma_2^2)).
$$
Then
\begin{eqnarray*}
2\hbox{Re}(\gamma_{\bar2}  \d_2 F(z_0))
&=&-{(A_1^2+2)  2\hbox{Re} (\gamma_{\bar2} E_2)\over 3 J}+7 {A_1^2  A_2 \over 18 J^2}
-{1\over 3J } \Big( 2\hbox{Re} \gamma_{\bar2}^2 E_{22}+A_2 2\hbox{Re} \gamma_{\bar2} E_{\bar2}\Big)\\
&&-{2\over 3}  E_{2\bar2}+{1\over 3} 2\hbox{Re} (\gamma_{\bar2}  E_{1\bar1 2})+{2\over 3}E_{1\1bar}
+2{A_2^2\over 9J^2} \\
&=&{-A_1^4-2A_1^2 +7A_1^2A_2-A_1^2A_2^2+4A_2^2 \over 18 J^2} +{A_1^4A_2+3A_1^2A_2\over 36 J^3} 2\hbox{Re}(\gamma_2^2)\\
&&-{1\over 3J} (2\hbox{Re}(\gamma_{\bar2}^2 E_{22})-{2\over 3} E_{2\bar2} +{2\over 3} E_{1\bar1}+{1\over 3} 2\hbox{Re} (\gamma_{\bar2}  E_{1\bar1 2}).
\end{eqnarray*}
Since
\begin{eqnarray*}
E_{1\bar1}-E_{2\bar2}
&=&-{A_1^2 A_2(\gamma_2^2+\gamma_{\bar2}^2)\over 6 J^3}+{A_1^2 +3A_1^4\over 12 J^2}
-\Big({ A_1^2 A_2 (\gamma_2^2+ \gamma_{\bar2}^2 ) \over 6 J^3}
-{ A_1^2  \over 12 J^2}
-{A_1^2  A_2^2 \over 12 J^2}.\Big)\\
&=&-{A_1^2 A_2(\gamma_2^2+\gamma_{\bar2}^2)\over 3 J^3}+{2A_1^2 +3A_1^4 +A_1^2A_2^2\over 12 J^2}
\\
\end{eqnarray*}
and 
\begin{eqnarray*}
-{1\over 3J} 2\hbox{Re} (\gamma_{\bar2}^2 E_{22}(z_0))&=&-{1\over 3J} 2 \hbox{Re} \Big(\gamma_{\bar2}^2(-{A_1^2 A_2^2 \gamma_{\bar2}^2 \over 6 J^3}+{A_1^2A_2 \over 3 J^2})\Big)\\
&=& {A_1^2 A_2^2  \over 18 J^4} 2\hbox{Re}(\gamma_2^4) -{A_1^2A_2 \over 9 J^3} 2\hbox{Re}(\gamma_2^2),
\end{eqnarray*}
one has
\begin{eqnarray*}
2\hbox{Re}(\gamma_{\bar2}  \d_2 F(z_0))
&=&{2A_1^4+7A_1^2A_2+4A_2^2 \over 18 J^2} +{A_1^4A_2-9A_1^2A_2\over 36 J^3} 2\hbox{Re}(\gamma_2^2)\\
&&+{A_1^2A_2^2 \over 18J^4} 2\hbox{Re}(\gamma_{\bar2}^4)+{1\over 3} 2\hbox{Re} (\gamma_{\bar2}  E_{1\bar1 2}).
\end{eqnarray*}
Notice that
\begin{eqnarray*}
2\hbox{Re} \gamma_{\bar2} \d_{2 1 \1bar } E(z_0)
&=&{(3A_1^2A_2^2+A_1^2+A_2^2+3A_1^4) \over 6 J^2} + A_1^2A_2^2{ 2\hbox{Re} \gamma_{\bar2}^4  \over 2 J^4} +{ (A_1^2-3)A_1^2 A_22\hbox{Re} \gamma_{\bar2} ^2\over 2 J^3} 
\end{eqnarray*} 
one has
\begin{eqnarray}
2\hbox{Re}(\gamma_{\bar2}  \d_2 F(z_0))
&=&{5A_1^4+3A_1^2A_2^2+7A_1^2A_2+A_1^2 +5A_2^2 \over 18 J^2} \nonumber\\
&&+{7A_1^4A_2-27A_1^2A_2\over 36 J^3} 2\hbox{Re}(\gamma_2^2)+4{A_1^2A_2^2 \over 18J^4} 2\hbox{Re}(\gamma_{\bar2}^4).\qquad
\end{eqnarray}
Therefore,
 \begin{eqnarray*}
\d_{1\bar1} F(z_0)&=&4 E E_{1\bar1}(z_0)-{E_{1\bar1}\over 3 J}+{E\over 3J^2} J_{1\bar1}\\
&& -{1\over 3J} ( 2 E_{1\bar1} +A_1 E_{\bar1 \bar1}+A_1 E_{11})+2\hbox{Re}(\gamma_{\bar2 }E_{1 \bar12})
+{2\hbox{Re}(\gamma_{\bar2} E_2) J_{1\bar1} \over 3 J^2}\\
&&+{1\over 3J}\Big(|\gamma_2|^2 E_{1\bar1 1\bar1}+ (1+A_1^2)E_{2\bar2}-2\hbox{Re} ( \gamma_{\bar2} E_{2 1 \bar1} )-A_12\hbox{Re}(\gamma_{\bar2} E_{2\bar1 \bar1})\Big)\\
&&-{1\over 3J^2}\Big(|\gamma_2|^2 E_{1\bar1}\Big) J_{1\bar1}\\
&&+2{A_1^2 +A_2^2 \over 9 J^3} J_{1\bar1} +{A_1^2 (1+A_1^2) \over 9J^3}-{A_2^2 |\gamma_2|^2\over 3J^4} J_{1\bar1}\\
&=&-{3+2A_1^2 \over 3J}  E_{1\bar1}(z_0)-{A_1^4 \over 36J^3} \\
&& -{A_1 \over 3J} (  E_{\bar1 \bar1}+E_{11})-{1\over 3J}(\gamma_2 E_{1 \bar1 \bar2}+\gamma_{\bar2} E_{1 \bar1 2})
+{2\hbox{Re}(\gamma_{\bar2} E_2) A_1^2 \over 3 J^2}\\
&&+{1\over 3J}\Big(|\gamma_2|^2 E_{1\bar1 1\bar1}+ (1+A_1^2)E_{2\bar2}-2\hbox{Re} ( \gamma_{\bar2} E_{2 1 \bar1} -A_1\gamma_{\bar2} E_{2\bar1 \bar1})\Big)\\
&&+{3A_1^4 -A_1^2A_2^2+A_1^2  \over 9 J^3}  \\
&=&-{3+2A_1^2 \over 3J}  E_{1\bar1}(z_0)+{11 A_1^4 \over 36J^3}+{A_1^2  \over 9J^3}-{A_1^2 A_2^2 \over 9J^3}  \\
&& -{A_1 \over 3J} (  E_{\bar1 \bar1}+E_{11})-{2\over 3J} 2\hbox{Re} (\gamma_2 E_{1 \bar1 \bar2})
+{2\hbox{Re}(\gamma_{\bar2} E_2) A_1^2 \over 3 J^2}\\
&&+{1\over 3} E_{1\bar1 1\bar1}+{1+A_1^2 \over 3J}E_{2\bar2}-{A_1\over 3J} 2\hbox{Re} (\gamma_{\bar2} E_{2\bar1 \bar1}).
\end{eqnarray*}
Notice that
\begin{eqnarray*}
L_1:&=&-{3+2A_1^2 \over 3J}  E_{1\bar1}(z_0)+{11 A_1^4 \over 36J^3}+{A_1^2  \over 9J^3}-{A_1^2 A_2^2 \over 9J^3}  \\
&=&-{3+2A_1^2 \over 3J} \Big( -{A_1^2 A_2(\gamma_2^2+\gamma_{\bar2}^2)\over 6 J^3}+{A_1^2 +3A_1^4\over 12 J^2}\Big)
+{11 A_1^4 \over 36J^3}+{A_1^2  \over 9J^3}-{A_1^2 A_2^2 \over 9J^3}  \\
&=& {3A_1^2A_2+2A_1^4A_2 \over 18 J^4} 2\hbox{Re}(\gamma_2^2)+ {-6A_1^6 -4A_1^2 A_2^2+A_1^2 \over 36 J^3}.
\end{eqnarray*}
Let
\begin{eqnarray*}
L_2:&=& -{A_1 \over 3J} (  E_{\bar1 \bar1}+E_{11})+{A_1^2 \over 3J^2} 2\hbox{Re}(\gamma_{\bar2} E_2) +{1+A_1^2 \over 3J}E_{2\bar2}.
\end{eqnarray*}

Notice that $\d_1E(z_0)=\d_1J(z_0)=\gamma_1(z_0)=\gamma_{i j\bar k}(z_0)=0$ and
$$
f_{11}(z_0)=-A_1^3|\gamma_2|^2+2A_1A_2(\gamma_2^2+A_1^2 \gamma_{\bar2}^2).
$$
Thus,
\begin{eqnarray*}
E_{11}(z_0)&=&-{f_{11}(z_0)\over 12 J^3} +A_1{f(z_0)\over 4 J^4} \\
&=&-{-A_1^3|\gamma_2|^2+2A_1A_2(\gamma_2^2+A_1^2 \gamma_{\bar2}^2)
\over 12 J^3}+{A_1^3  \over 4 J^2}\\
&=&-{A_1A_2(\gamma_2^2+A_1^2 \gamma_{\bar2}^2)
\over 6  J^3}+{A_1^3  \over 3 J^2}.
\end{eqnarray*}
Therefore,
$$
E_{11}(z_0)+E_{\bar1\bar1}(z_0)=-{A_1A_2 (1+A_1^2) (\gamma_2^2+ \gamma_{\bar2}^2)
\over 6  J^3}+{2A_1^3  \over 3 J^2},
$$

\begin{eqnarray*}
L_2:&=& -{A_1 \over 3J} ( -{A_1A_2 (1+A_1^2) (\gamma_2^2+ \gamma_{\bar2}^2)
\over 6  J^3}+{2A_1^3  \over 3 J^2})\\
&&+{A_1^2 \over 3J^2} \Big( {A_1^2 A_2 2\hbox{Re} \gamma_2^2 \over 12J^2}-{A_1^2\over 6J} \Big)\\
&& +{1+A_1^2 \over 3J}\Big({ A_1^2 A_2 (\gamma_2^2+ \gamma_{\bar2}^2 ) \over 6 J^3}
+{ -4A_1^4 +3A_1^2A_2^2 -A_1^2  \over 12 J^2}\Big)\\
&=& {2 A_1^2A_2(1+A_1^2) +A_1^4A_2+2A_1^2A_2(1+A_1^2)\over 36 J^4}  2\hbox{Re}(\gamma_2^2) \\
&&+{ -4A_1^6 +3A_1^4A_2^2 +3A_1^2A_2^2-15A_1^4-A_1^2 \over 36 J^3}.
\end{eqnarray*}
Moreover,
\begin{eqnarray*}
L_3:&=& -{2\over 3J} 2\hbox{Re}(\gamma_{\bar2} E_{2 1\bar1})-{A_1\over 3J} 2\hbox{Re}(\gamma_{\bar2} E_{2 \bar1\bar1})\\
&=&-{2\over 3J} \Big({(3A_1^2A_2^2+A_1^2+A_2^2+3A_1^4) \over 6 J^2} + A_1^2A_2^2{ 2\hbox{Re} \gamma_{\bar2}^4  \over 2 J^4} 
+{ (A_1^4 A_2 -3A_1^2 A_2) \over 2 J^3} 2\hbox{Re}( \gamma_{\bar2} ^2)
\Big)\\
&&-{A_1\over 3J}\Big({(2A_1 A_2^2 
+ 4A_1^3+6A_1^3A_2^2 -A_1^4A_2)  \over 6 J^2}\\
&&\qquad\qquad +{(-8A_1^3 A_2- 4A_1 A_2)   \over 12 J^3} 2\hbox{Re}(\gamma_2^2) +{2A_1A_2^2\over 4J^4} 2\hbox{Re}( \gamma_{\bar2}^4)\Big) \\
&=& -{6A_1^4 A_2^2 -A_1^5A_2+8A_1^2A_2^2+2A_1^2+2A_2^2+10A_1^4 \over 18 J^3}\\
&&+{-A_1^4 A_2+10 A_1^2 A_2   \over 9 J^3} 2\hbox{Re}(\gamma_2^2) -{A_1^2A_2^2\over 2 J^5} 2\hbox{Re}( \gamma_{\bar2}^4).
\end{eqnarray*}
Therefore,
\begin{eqnarray*}
F_{1\bar1}&=& L_1+L_2+L_3+{E_{1\bar1 1 \bar1} \over 3}\\
&=& {-10A_1^6 +A_1^5A_2 -9A_1^4A_2^2 -35 A_1^4 -17 A_1^2 A_2^2 -4A_1^2-4A_2^2\over 36 J^3} \\
&& +{3A_1^4 A_2+47 A_1^2 A_2   \over 36 J^3} 2\hbox{Re}(\gamma_2^2) -{A_1^2A_2^2\over 2 J^5} 2\hbox{Re}( \gamma_{\bar2}^4)\\
&&-{6A_1^6 +2A_1^4 A_2^2 +12A_1^4+3A_1^3+4A_1^2A_2^2 \over 18 J^3}\\
&&+{(8A_1^3A_2 +2A_1^2A_2(1+A_1^2)) \over 12J^4} 2\hbox{Re}(\gamma_2^2)\\
&=& {-16A_1^6 +A_1^5A_2 -11A_1^4A_2^2 -47 A_1^4 -21 A_1^2 A_2^2 -3A_1^3 -4A_1^2-4A_2^2\over 36 J^3} \\
&& +{5A_1^4 A_2+8A_1^3A_2+49 A_1^2 A_2   \over 36 J^3} 2\hbox{Re}(\gamma_2^2) -{A_1^2A_2^2\over 2 J^5} 2\hbox{Re}( \gamma_{\bar2}^4)\\
\end{eqnarray*}

\subsection{End of the proof of Theorem 1.1}

Since $\gamma_1(z_0)=0$ and $E(z_0)=-A_1^2/(12J)$, one has
\begin{eqnarray*}
{\cal A}(z_0)
&=&6 E(F-E^2) +(4F+2E^2) {1 \over J} +{2E \over J} 2\hbox{Re} (\gamma_{\bar2} E_2) \\
&&+ {1\over J} \Big( 2 2\hbox{Re} (\gamma_{\bar2} F_2)   -2E |\gamma_2|^2 E_{1\bar1} -(E_{1\bar1}+E_{2\bar2})
-|\gamma_2|^2 F_{1\bar1}\Big)\\
&&+{1\over 3J} {A_1^2 \over J} 2\hbox{Re} (\gamma_{\bar2} E_2)+{4E \over 3 J}  {A_1^2 \over J} \\
&=&-{A_1^2 \over 2J} (F-E^2) +{2\over J} (2F-7E^2)  -{A_1^2 \over 6J^2 } 2\hbox{Re} (\gamma_{\bar2} E_2) \\
&&+ {2 \over J} 2\hbox{Re} (\gamma_{\bar2} F_2)   +{A_1^2-6 \over 6 J} E_{1\bar1} -{E_{2\bar2} \over J} 
-F_{1\bar1}.\\
\end{eqnarray*}
Since
\begin{eqnarray*}
K_1:&=&-{A_1^2 \over 2J} (F-E^2)+{2\over J}(2F-7E^2)\\
&=&-{A_1^2 \over 2J}\Big( {7 A_1^4 \over 72 J^2}  -A_1^2A_2 {\gamma_2^2+\gamma_{\bar2}^2\over 9 J^3} -{A_1^4\over 144J^2}\Big)\\
&&+{2\over J} \Big( {7 A_1^4 \over 36 J^2} -2A_1^2A_2 {\gamma_2^2+\gamma_{\bar2}^2\over 9 J^3} -{7A^4\over 144 J^2}\Big)\\
&=&-{13 A_1^6\over 288 J^3}+{7 A_1^4\over 24J^3}-{A_1^2A_2(A_1^2-8)\over 18 J^4} (\gamma_2^2+\gamma_{\bar2}^2)
\end{eqnarray*}
and
\begin{eqnarray*}
K_2:&=&-{A_1^2\over 6J^2} 2\hbox{Re}(\gamma_{\bar2} E_2)+{2\over J} 2\hbox{Re}(\gamma_{\bar2} F_2)\\
&=&-{A_1^2\over 6J^2} \Big({A_1^2 (A_2 2\hbox{Re} \gamma_2^2-2J) \over 12J^2}  \Big)\\
&&+ {2\over J}\Big({5A_1^4+3A_1^2A_2^2+7A_1^2A_2+A_1^2 +5A_2^2 \over 18 J^2} \nonumber\\
&&\qquad \qquad +{7A_1^4A_2-27A_1^2A_2\over 36 J^3} 2\hbox{Re}(\gamma_2^2)+4{A_1^2A_2^2 \over 18J^4} 2\hbox{Re}(\gamma_{\bar2}^4)\Big)\\
 &=&\Big({27A_1^4A_2 -108A_1^2 A_2\over 72 J^4}\Big) 2\hbox{Re}(\gamma_2^2)
 +{4 A_1^2 A_2^2 \over 9J^5} 2\hbox{Re}(\gamma_2^4)\\
 &&+{20A_1^4 +12 A_1^2A_2^2+28A_1^2A_2+ 5A_1^2+20A_2^2 \over 36 J^3}.
\end{eqnarray*}
Therefore,
\begin{eqnarray*}
K_1+K_2
&=&{23A_1^4A_2 -76 A_1^2 A_2\over 72 J^4} 2\hbox{Re}(\gamma_2^2)
 +{4 A_1^2 A_2^2 \over 9J^5} 2\hbox{Re}(\gamma_2^4)\\
 &&-{13 A_1^6\over288 J^3} +{61A_1^4\over 72 J^3} +{12 A_1^2A_2^2+28A_1^2A_2+ 5A_1^2+20A_2^2 \over 36 J^3}.
\end{eqnarray*}

Let
$$
K_3:={A_1^2-6\over 6J} E_{1\bar1}-{E_{2\bar2}\over J}.
$$
Then 
\begin{eqnarray*}
K_3
&=&{A_1^2-6\over 6J} \Big(  -{A_1^2 A_2(\gamma_2^2+\gamma_{\bar2}^2)\over 6 J^3}+{A_1^2 +3A_1^4\over 12 J^2}\Big)\\
&&-{1 \over J}\Big(  -{A_1^2 A_2(\gamma_2^2+\gamma_{\bar2}^2)\over 6 J^3}+{A_1^2 +3A_1^4\over 12 J^2}\Big) \\
&=& {-A_1^4A_2+12A_1^2A_2 \over 36 J^4} (\gamma_2^2+\gamma_{\bar2}^2)+{3A_1^6-35A_1^4-12A_1^2 \over 72 J^3}.
\end{eqnarray*}
Therefore,
\begin{eqnarray*}
\lefteqn{ K_1+K_2+K_3}\\
&=&{21A_1^4A_2 -52 A_1^2 A_2\over 72 J^4} 2\hbox{Re}(\gamma_2^2)
 +{4 A_1^2 A_2^2 \over 9J^5} 2\hbox{Re}(\gamma_2^4)\\
 &&-{A_1^6\over288 J^3} +{13A_1^4+12 A_1^2A_2^2+28A_1^2A_2-A_1^2+20A_2^2 \over 36 J^3}.
\end{eqnarray*}
Thus,
\begin{eqnarray*}
{\cal A}&=&K_1+K_2+K_3-F_{1\bar1}\\
&=&{21A_1^4A_2 -52 A_1^2 A_2\over 72 J^4} 2\hbox{Re}(\gamma_2^2)
 +{4 A_1^2 A_2^2 \over 9J^5} 2\hbox{Re}(\gamma_2^4)\\
 &&-{A_1^6\over288 J^3}+{13A_1^4+ 12 A_1^2A_2^2+28A_1^2A_2-A_1^2+20A_2^2 \over 36 J^3}\\
 && -{-16A_1^6 +A_1^5A_2 -11A_1^4A_2^2 -47 A_1^4 -21 A_1^2 A_2^2 -3A_1^3 -4A_1^2-4A_2^2\over 36 J^3} \\
&& -{5A_1^4 A_2+8A_1^3A_2+49 A_1^2 A_2   \over 36 J^3} 2\hbox{Re}(\gamma_2^2) +{A_1^2A_2^2\over 2 J^5} 2\hbox{Re}( \gamma_{\bar2}^4)\\
&=&  {(127/8)A_1^6 -A_1^5A_2 +11A_1^4A_2^2 +60 A_1^4 +33 A_1^2 A_2^2 +3A_1^3 +28A_1^2A_2 +3A_1^2+24A_2^2\over 36 J^3} \\
&& +{11A_1^4 A_2-16A_1^3A_2-150A_1^2 A_2   \over 72 J^3} 2\hbox{Re}(\gamma_2^2) +{17A_1^2A_2^2\over 18 J^5} 2\hbox{Re}( \gamma_{\bar2}^4).
\end{eqnarray*}
Write
$$
\gamma_2(z_0)^2=|\gamma_2|^2 e^{i\theta}=J(z_0) e^{i\theta}.
$$
Then
\begin{eqnarray*}
\lefteqn{36J^3 {\cal A}(z_0)}\\
&=&  (127/8)A_1^6 -A_1^5A_2 +11A_1^4A_2^2 +60 A_1^4 +33 A_1^2 A_2^2 +3A_1^3 +28A_1^2A_2 +3A_1^2+24A_2^2 \\
&& +(11A_1^4 A_2-16A_1^3A_2-150A_1^2 A_2 ) \cos (\theta) +68 A_1^2A_2^2 \cos(2\theta)\\
 &=&0.
\end{eqnarray*}
For 
$$
z_0\in X=\{z_2\in \CC: |z_2|^2+A_2\hbox{Re}(z_2^2)=1\}
$$
$\theta $ is not constant. This implies that $J^3 {\cal A}(z_0)=0$ if and only if
$$
 (127/8)A_1^6 -A_1^5A_2 +11A_1^4A_2^2 +60 A_1^4 +33 A_1^2 A_2^2 +3A_1^3 +28A_1^2A_2 +3A_1^2+24A_2^2=0,
 $$
 $$
11A_1^4 A_2-16A_1^3A_2-150A_1^2 A_2 =0
$$
and
$$
68 A_1^2A_2^2 =0.
$$
This couping with $A_1\ge 0$ imply that
$$
A_1=A_2=0.
$$

Theorem \ref{main} is completely proved now.

\section{Ramadanov Conjecture via Tresse invariants}

 \subsection{The method of associated differential equations}
 The study of the relationship between the geometry of 
 real hypersurfaces in $\CC^{2}$ and 2nd order ODEs 
  \begin{equation}\label{wzz}
 w\rq{}\rq{}=\Phi(z,w,w\rq{}).
 \end{equation}
 goes back to Segre \cite{segre}
 and Cartan \cite{cartan},
 see also Webster \cite{webster}.
More generally,  the geometry of a real hypersurface in $\CC^{n},\,n\geq 3$, is related to
 that of a complete second order system of PDEs
\begin{equation}\label{wzkzl}
w_{z_kz_l}=\Phi_{kl}(z_1,...,z_n,w,w_{z_1},...,w_{z_{n-1}}),\quad \Phi_{kl}=\Phi_{lk},\quad k,l=1,...,n-1,
\end{equation}
  Moreover, {\em in the real-analytic case},
  this relation becomes 
more
explicit 
by means of the Segre family.
Namely,
to any real-analytic Levi-nondegenerate hypersurface $M\subset\CC^{n},\,n\geq 2$, one can uniquely associate a holomorphic ODE \eqref{wzz} ($n=2$) or a holomorphic PDE system \eqref{wzkzl} ($n\geq 3$),
whose solutions are precisely the Segre varieties.
The Segre family of $M$ plays a role of a 
``mediator'' between the hypersurface and the associated differential equations.  
It is, however, only defined in the real-analytic case.
For recent work on associated differential equations in the degenerate setting, see e.g. the papers  
  \cite{divergence, nonminimalODE, nonanalytic} of the first author with Lamel and Shafikov.

  The associated differential equation procedure is particularly simple in the case of a Levi-nondegenerate hypersurface in $\CC^{2}$. In this case the Segre family is an
 anti-holomorphic 
 2-parameter family of complex holomorphic curves. 
 It then follows from the standard ODE theory that there exists a unique 2nd order ODE \eqref{wzz}, for which the Segre varieties are precisely the graphs of solutions. This ODE is called \it the associated ODE. \rm

In general, both right hand sides in \eqref{wzz} and \eqref{wzkzl} appear as functions determining the $2$-jet of a Segre variety 
at a given point
as a function of the $1$-jet at the same point.
More explicitly, we use coordinates
$$
	(z,w)=(z_1, \ldots, z_{n-1},w) \in \CC^{n-1}\times \CC{} = \CC^{n}.
$$ 
Fix a 
 real-analytic
hypersurface
 $M\subset\CC^{n}$
 passing through the origin, 
 and choose a 
sufficiently small neighborhood $U$
 of the origin.
  When $M$ is Levi-nodegenerate, we can associate a 
{\em complete second order system of holomorphic PDEs}
 to $M$,
which is uniquely determined by the condition that the differential equations are satisfied by all the
graphing functions $h(z,\zeta) = w(z)$ of the
Segre varities $\{Q_\zeta\}_{\zeta\in U}$ of $M$ in a
neighbourhood of the origin.

To be more precise, we consider a
so-called {\em  complex defining
 equation } (see, e.g., \cite{ber}),
$$
w=\rho(z,\bar z,\bar w),
$$ 
of $M$ near the origin, which one
obtains by substituting 
$$u=\frac{1}{2}(w+\bar w),\,v=\frac{1}{2i}(w-\bar w)$$ 
into 
a real-analytic defining equation and
solving for $w$ as function of $(z,\bar z, \bar w)$
by the implicit function theorem.
 The Segre
variety $Q_x$ of a point 
$$x=(a,b)\in U,\quad
a\in\CC^{n},\,b\in\CC{},$$ 
is  now given
as the graph of the function
\begin{equation} 
\label{segredf}
w (z)=\rho(z,\bar a,\bar b), 
\end{equation}
where we slightly abuse the notation
using the same letter $w$ for 
both the last coordinate and
the function $w(z)$ defining a Segre variety.
Differentiating \eqref{segredf} we obtain
\begin{equation}\label{segreder} 
	w_{z_j}=\rho_{z_j}(z,\bar a,\bar b),
	\quad
	j=1,\ldots,n. 
\end{equation}
Considering \eqref{segredf} and \eqref{segreder}  as a holomorphic
system of equations with the unknowns $\bar a,\bar b$, 
in view of the Levi-nondegeneracy of $M$,
an
application of the implicit function theorem yields holomorphic functions
 ${\cal C}_1,...,{\cal C}_n, B$ such that
 \eqref{segredf} and \eqref{segreder} are solved by
$$
	\bar a_j={\cal C}_j(z,w,w'),\quad
	\bar b=B(z,w,w'),
$$
where we write
$$
	w' = (w_{z_1},  \ldots, w_{z_n}).
$$
The implicit function theorem applies here because the
Jacobian of the system coincides with the Levi determinant of $M$
for $(z,w)\in M$ (\cite{ber}). Differentiating \eqref{segredf} twice
and substituting the above solution for $\bar a,\bar b$ finally
yields
\begin{equation}\label{segreder2}
w_{z_kz_l}=\rho_{z_kz_l}(z,{\cal C} (z,w,w'),
B(z,w,w'))=:\Phi_{kl}(z,w,w'),
\quad
k,l=1, \ldots, n,
\end{equation}
or, more invariantly,
i.e.\ independent of the coordinate choice,
\begin{equation}\label{segreder2'}
	j^2_{(z,w)} Q_x = \Phi(x, j^1_{(z,w)} Q_x).
\end{equation}
Now \eqref{segreder2}
(or \eqref{segreder2'})
is the desired complete system of holomorphic second order PDEs
denoted by $\mathcal E = \mathcal{E}(M)$.
\smallskip

\noindent{\bf Definition.} We call $\mathcal E = \mathcal{E}(M)$  \it the system of PDEs 
 associated with $M$. \rm  

\subsection{Connecting differential operators in the $\rho$ and the $\Phi$ functions} 

It is convenient to consider the function $\Phi$ introduced above as one defined on the bundle $J^{1,n-1}$ of complex hypersurfaces in $\CC^{n}$ (which can be identified with the projectivized cotangent bundle of $\CC^n$). We consider the canonical section of this bundle over $M$
defined by the distribution of complex tangent spaces:
\beq
\label{th-section}
s_M^1\colon M\to J^{1,n-1}, 
\quad
q\mapsto (q, [T^\CC_q M]).
\eeq
Fixing a system of local holomorphic  coordinates
in $\CC^{n-1}\times\CC{}$ yields a (local) trivialization of the bundle $J^{1,n-1}$, and we can write
\beq
\label{th-def}
s_M^1(q) = (q, \th(q))
\in M\times \CC^{n-1} \subset M\times \CC \mathbb P^{n-1},
\eeq
where $\th=(\th_1,\ldots,\th_{n-1})$.
It is easy to express $\th$
in terms of a local defining function $\rho$ of $M$
with $\rho_w\ne 0$
in local holomorphic coordinates
\beq
\label{th-coor}
\th_j = -\frac{\rho_{z_j}}{\rho_w}, 
\ \  1\le j\le n-1,
\ \ 
(z,w)=(z_1,\ldots,z_{n-1},w)\in \CC^{n-1}\times \CC{}.
\eeq

We consider the image $M_J:=s^1_M(M)\subset J^{1,n-1}$. As observed by Webster, $M_J$ is a totally real submanifold in $J^{1,n-1}$ diffeomorphic to $M$. We can then treat $\Phi$ as a function on the above trivialization of $J^{1,n-1}$. Its restriction on $M_J$ (or equivalently on $M$) we call {\it the $\Phi$ function for $M$}. For an analytic $M$, the $\Phi$ function can be considered as an analytic function on either $M_J$ or $M$ itself. The function $\Phi$ can be calculated (see \cite{KZ1,KZ2})
in terms of a local defining function $\rho$ of $M$
with $\rho_w\ne 0$
by the formula
\begin{equation}\label{Phi00}
\Phi = (\Phi_{ij}),
\quad
	\Phi_{ij}=\frac{1}{(\rho_w)^3}
\begin{vmatrix} 0 & \rho_{z_j} & \rho_w\\ \rho_{z_i} & \rho_{z_iz_j} & \rho_{z_iw} \\ \rho_w & \rho_{z_jw} & \rho_{ww} \end{vmatrix},
 \quad i,j=1,...,n-1.
\end{equation}

As shown in \cite{KZ2}, one can then express (for $M$ analytic and even in a more general setup, see \cite{KZ2}) differential operators acting on $\Phi$ considered as a function on $J^{1,n-1}$ as those acting on $\Phi$ considered as a function on $M$, and subsequently as certain differential operators acting on the real defining function $\rho$. We describe below the procedure. 

{\bf Let us consider from now on the case $n=2$}, and introduce the symbolic notation:
\begin{equation}\label{Jzwab}
	J(z,w,a,b):=
\begin{vmatrix} 0 & \rho_{a} & \rho_b\\ \rho_{z} & \rho_{za} & \rho_{zb} \\ \rho_w & \rho_{wa} & \rho_{wb} \end{vmatrix}.
\end{equation}
Note that
$$J_0:=J(z,w,\bar z,\bar w)$$
is proportional to the Levi determinant of $M$ and hence is nonvanishing, and also that, from the above, 
\beq\label{Phi=}\Phi=\frac{J(z,w,z,w)}{(\rho_w)^3}.\eeq

Following \cite{KZ2}, the operator $L_\xi$ on $M$, corresponding to the differentiation $\frac{\partial}{\partial\xi}$, it is obtained from the conditions
$$L_\xi(z)=L_\xi(w)=L_\xi(\bar z)=L_\xi(\bar w)=L_\xi(\bar\xi)=0,\quad L_\xi(\xi)=1$$
(where we subject $\xi=\theta$, according to the defining equations of $M_J$). Using the above six equations and the fact that $L_\xi$ is a section of $\CC{}TM$, a direct calculation gives: 
\beq\label{Lxi}
L_\xi=\frac{(\rho_w)^2}{J_0}(\rho_{\bar w}\partial_{\bar z}-\rho_{\bar z}\partial_{\bar w}).
\eeq
Similarly, we find sections $L_z$ and $ L_w$ of $\CC{}TM$, corresponding respectively to the operators $\partial_{z}$ and $\partial_w$, and satisfying, respectively,
$$L_z(\xi)=L_z(w)=L_z(\bar z)=L_z(\bar w)=L_z(\bar\xi)=0,\quad L_z(z)=1$$
and
$$L_w(\xi)=L_w(z)=L_w(\bar z)=L_w(\bar w)=L_w(\bar\xi)=0,\quad L_w(w)=1,$$
where we again subject $\xi=\theta$  according to the defining equations of $M_J$.
 Introducing the notations
$$J_1:=J(z,w,z,\bar w), \quad J_2:=J(z,w,\bar z,z), 
$$
$$
 J_3:=J(z,w,w,\bar w)\quad\hbox{and}\quad  J_4:=J(z,w,\bar z,w),$$
we argue as above and compute:
\beq\label{LzLw}
L_z=\partial_z-\frac{J_1}{J_0}\partial_{\bar z}-\frac{J_2}{J_0}\partial_{\bar w}, \quad 
L_w=\partial_w-\frac{J_3}{J_0}\partial_{\bar z}-\frac{J_4}{J_0}\partial_{\bar w}.
\eeq
Now, {\it formulas \eqref{Phi=},\eqref{Lxi},\eqref{LzLw} allow for expressing any holomorphic differential polynomial in $\Phi$, viewed  as a function on $J^{1,n-1}$, in terms of the operators $L_z,L_w,L_\xi$ and the real defining function $\rho$ of $M$.}

Precisely, for any holomorphic function $\Psi(z,w,\xi)$ in an open neighborhood of $M_J$ and its restriction $\psi:=\Psi|_{M_J}$, it holds that:
\beq\label{transfer}\Psi_z=L_z(\psi),\quad \Psi_w=L_w(\psi),\quad \Psi_\xi=L_\xi(\psi),
\eeq
where we identify $M_J$ with $M$ here.

\subsection{Obstruction function via Tresse invariants}

In this section, we compare two families of relative CR invariants associated with a Levi nondegenrate CR hypersurface in $\CC^{2}_{z,w}.$ Recall that by a relative invariant we mean a function on a manifold being multiplied by a nonzero factor after an arbitrary (in our case CR) coordinate change.  

The first family is formed of the coefficients of the Chern-Moser normal form of a CR hypersurface introduced in \cite{CM74}. The normal form at a point $x\in M$ looks as:
$$\Re(w)=z\overline{z}+\sum_{a\geq 2,b\geq 2,c\geq 0 }N^c_{a,b}z^a\overline{z}^b\Im(w)^c,$$
where $N^c_{2,2}=N^c_{2,3}=N^c_{3,3}=0$ and $N^c_{a,b}=\overline{N^c_{b,a}}$ for all $c$. With the CR weight system setup as
$$
z\mapsto \lambda z,\quad \overline{z}\mapsto \mu \overline{z},\quad w\mapsto \lambda\mu \overline{w},
$$
$N^c_{a,b}$ transforms as $\lambda^{a+c-1}\mu^{b+c-1}$ that is has CR weight $(a+c-1,b+c-1)$ as a relative invariant.

According to Graham \cite[Remark 4.13]{Gr87}, {\em the value of the obstruction function of $M$ at $x$ we are interested in is (proportional to) the coefficient $N^0_{4,4}$ that is a relative invariant of the CR weight $(3,3)$.}

The second family of invariants consists of the {\em Tresse invariants} \cite{Tr96} of the associated second order ODE $$w''=\Phi(z,w,p),\,p=w'.$$ Let us denote $$\Phi^c_{a,b}=\partial_z^a\partial_w^b\partial_p^c\Phi(z,w,p),\,\, \Phi_{a,b}=\partial_z^a\partial_w^b\Phi(z,w,p),\,\,\Phi^c=\partial_p^c\Phi(z,w,p).$$ The Tresse invariants are generated by two basic invariants of order $4$:
 one is
$$I:=\Phi^4,$$
and other one is
\begin{equation} \label{Hformula}\begin{aligned}  H:=&\Phi^2_{2,0}+(\Phi)^2 \Phi^4+2\Phi\Phi^3_{0,1}p-\Phi^2_{0,1}p\Phi^1+\Phi_{0,1}\Phi^3p+\Phi^2_{0,2}p^2\\
&-3\Phi^2_{0,1}\Phi+2\Phi^3_{1,0}\Phi
 +\Phi^3\Phi_{1,0}-\Phi^1\Phi^2_{1,0}+4\Phi^1\Phi^1_{0,1}-3\Phi^2\Phi_{0,1}\\
 &+2p\Phi^2_{1,1}-4\Phi^1_{0,2}p-4\Phi^1_{1,1}+6\Phi_{0,2} \end{aligned} \end{equation}
together with invariant derivations $\Delta_z,\Delta_w,\Delta_p$ (on generic strata $H$ and the derivations are enough). According to \cite{Kr09} the weights of these invariants and operations are given by a weight system setup as
\beq\label{scaling}
z\mapsto \tilde \lambda z,\quad w\mapsto \tilde \mu w,\quad p=w'\mapsto \frac{\tilde \mu}{\tilde \lambda}p, \quad q=w''\mapsto \frac{\tilde \mu}{\tilde \lambda^2}q.
\eeq
One can see from \eqref{scaling} that $I$ has the weight $(-2,3)$, $H$ has the weight $(2,1)$, $\Delta_z$ has the weight $(1,0)$, $\Delta_w$ has the weight $(0,1)$ and $\Delta_p$ has the weight $(-1,1).$ Let us start by determining the CR weights of these invariants.

\begin{proposition}\label{weightcheck}
The Tresse invariant of weight $(r,s)$ is relative CR invariant of CR weight $(r+s,s)$. Moreover, the Tresse invariant of weight $(r,s)$ of the lowest order $k$ vanishes if and only if $N^{2s-k+r}_{1-s+k,1-s+k-r}$ vanishes.
\end{proposition}
\begin{proof}
Note that the CR weights define the weighting system:
$$
z\mapsto \lambda z,\quad w\mapsto \lambda \mu w,\quad p\mapsto \mu p
$$
(recall that $p=dw/dz$).
Thus $ \tilde  \lambda=\lambda$ and $\tilde \mu= \lambda \mu$ and the Tresse invariant of weigth $(r,s)$ is relative CR invariant of CR weight $(r+s,s)$.

Let us take the order of the invariant into account: $N^c_{a,b}$ is of order $a+b+c$, while $\Phi$ is obtained by two derivations, i.e., $a+b+c-2=k$. Among the Tresse invariant of weigth $(r,s)$ there is unique one of the lowest order $k$ and thus it is  $N^c_{a,b}$ with $a+b+c-2=k,a+c-1=r+s,b+c-1=s$ 
up to nowhere vanishing absolute invariant, i.e., it vanishes if and only if the Tresse invariant vanishes.
\end{proof}

Let us remark that if one finds out which Tresse invariants of weight $(r,s)$ of lower order than $k$ pose as correction terms for the Tresse invariant of weigth $(r,s)$ of general order $k$ to obtain $N^c_{a,b}$, then the Proposition \ref{weightcheck} could be extended to general order invariants.

Consequently, for the vanishing of the obstruction function, we are looking for Tresse invariant of weight $(0,3)$ of order $6$, which we  find to be $\Omega^4_{2,0}$ in the notations of \cite{Tr96}.

In this way, we have obtained
\begin{theorem}\label{reduction}
The obstruction function of a real-analytic Levi nondegenrate CR hypersurface $M\subset\CC^{2}$ vanishes if and only if for the associated ODE $w''=\Phi(z,w,w')$ is holds that
\begin{multline*}
0=\Omega^4_{2,0}=(\partial_p)^2H=\Phi^4_{2,0}+
4\Phi^5\Phi^1\Phi+\Phi^5\Phi_{0,1}p+2\Phi^4 (\Phi^1)^2+2\Phi^4\Phi^2\Phi+2\Phi^4p\Phi^1_{0,1}\\
+\Phi^6 (\Phi)^2+3\Phi^1 \Phi^4_{0,1}p
+2\Phi^5_{0,1}\Phi p+\Phi^4_{0,2}p^2+\Phi^5\Phi_{1,0}-\Phi^4\Phi_{0,1}\\
+2\Phi^4\Phi^1_{1,0}+3\Phi^4_{1,0}\Phi^1+2\Phi^5_{1,0}\Phi+\Phi^4_{0,1}\Phi+2\Phi^4_{1,1}p.
\end{multline*}
\end{theorem}
\begin{proof}Following  \cite{Tr96}, the Tresse invariant $\Omega^4_{2,0}$ is the invariant of weight $(0,3)$ of order $6$ and this is the lowest order invariant among invariants  of weight $(0,3)$. Therefore, $\Omega^4_{2,0}$ vanishes if and only if the obstruction function vanishes by Proposition \ref{weightcheck}. Following \cite{Tr96}, we see that $\Omega^4_{2,0}=(\partial_p)^2H$ and obtain the claimed formula by differentiating \eqref{Hformula}.
\end{proof}

\medskip

\noindent{} {\bf Remark.} In fact, one can deduce Theorem \ref{reduction} by using the invariance of $\Omega^4_{2,0}$ under coordinate changes, and then checking the coincidence of $\Omega^4_{2,0}$ and the obstruction function in the normal form coordinates at a point directly.

\medskip

\subsection{Tresse invariants of ellipsoids}

In this section, we apply the outcomes of the preceding two sections in the case when $M$ is a real ellipsoid. Ellipsoids in $\CC^{2}$ form a two parameter family of CR hypersurfaces  given by:
$$
0=\rho(z,w,\overline{z},\overline{w})=z\overline{z}+w\overline{w}+\frac{A_1}{2}(z^2+\overline{z}^2)+\frac{A_2}{2}((w^2+\overline{w}^2))-1 
$$
parametrized by $a,b>0.$ 

According to \eqref{Phi=}, the corresponding $\Phi$ function is
$$\Phi:=\frac{1}{(\partial_w \rho)^3}J(z,w,z,w)=-\frac{A_1^2 A_2 z^2+A_1 A_2^2 w^2+2 A_1 A_2 w\overline{w}+2 A_1 A_2 z\overline{z}+A_1\overline{w}^2+A_2\overline{z}^2}{(A_2 w+\overline{w})^3}.$$

For the basic differentiations we have, in view of \eqref{transfer}, we have:
\begin{eqnarray*}
\partial_p\sim L_\xi&=&\frac{(\partial_w \rho)^2}{J(z,w,\overline{z},\overline{w})}((\partial_{\overline{w}} \rho)\partial_{\overline{z}}-(\partial_{\overline{z}} \rho)\partial_{\overline{w}})\\
&=&\frac{-(A_2 w+\overline{w})^2 ((A_2 \overline{w}+w)\partial_{\overline{z}}-(A_1\overline{z}+z)\partial_{\overline{w}})}{A_1^2 z\overline{z}+A_2^2 w\overline{w}+A_1\overline{z}^2+A_1z^2+A_2\overline{w}^2+A_2 w^2+w\overline{w}+z\overline{z}},\\
\end{eqnarray*}

\begin{eqnarray*}
\lefteqn{\partial_z\sim L_z}\\
&=&\partial_z-\frac{J(z,w,z,\overline{w})}{J(z,w,\overline{z},\overline{w})}\partial_{\overline{z}}-\frac{J(z,w,\overline{z},z)}{J(z,w,\overline{z},\overline{w})}\partial_{\overline{w}}\\
&=&
\partial_z-\frac{(A_1 A_2^2 w\overline{w}+A_1 A_2\overline{w}^2+A_1 A_2 w^2+A_1w\overline{w}+(A_1z+\overline{z})^2)\partial_{\overline{z}}-(A_2 w+\overline{w})\overline{z}(A_1^2-1)\partial_{\overline{w}}}{A_1^2 z\overline{z}+A_2^2 w\overline{w}+A_1\overline{z}^2+A_1z^2+A_2 \overline{w}^2+A_2 w^2+w\overline{w}+z\overline{z}}\\
\end{eqnarray*}
and
\begin{eqnarray*}
\lefteqn{\partial_w\sim L_w}\\
&=&\partial_w-\frac{J(z,w,w,\overline{w})}{J(z,w,\overline{z},\overline{w})}\partial_{\overline{z}}-\frac{J(z,w,\overline{z},w)}{J(z,w,\overline{z},\overline{w})}\partial_{\overline{w}}\\
&=&
\partial_w-\frac{((A_1 z+\overline{z})\overline{w}(1-A_2^2))\partial_{\overline{z}}+(A_1^2 A_2 z\overline{z}+A_1 A_2 \overline{z}^2+A_1 A_2 z^2+A_2 z\overline{z}+(A_2 w +\overline{w})^2)\partial_{\overline{w}}}
{A_1^2 z\overline{z}+A_2^2 w\overline{w}+A_1\overline{z}^2+A_1 z^2+b\overline{w}^2+A_2 w^2+w\overline{w}+z\overline{z}}.
\end{eqnarray*}
Let 
$$
H_0:=\frac{-3( A_1 A_2^2 {w}^{2}+ ( A_1^{2}{z}^{2}+
 ( 2w{ \overline{w}}+2z{ \overline{z}} ) A_1+{{ \overline{z}}}^{2} ) A_2+A_1
{{ \overline{w}}}^{2} ) } {( w{ \overline{w}}A_2^{2}+A_2 ( {{ \overline{w}}}^
{2}+{w}^{2} ) +A_1^{2}z{ \overline{z}}+A_1 ( {{ \overline{z}}}^{2}+{z}^{2}
 ) +w{ \overline{w}}+z{ \overline{z}} ) ^{5} ( \overline{w}+{ \overline{w}}
 )}.
 $$
We can now use those formulas to compute the Tresse invariant $H$ {\it on $M$}, Using
\eqref{Hformula} and also substituting $p=-\frac{\partial_z \rho}{\partial_w \rho}$, we  obtain with $a=A_1$ and $b=A_2$:
\begin{footnotesize}
\begin{multline*}
H=H_0\cdot
\Bigl [  ( {a}^{2}{{ \overline{w}}}^{4}{w}^{2}+4{{ \overline{w}}}^{2}{w}^
{4} ) {b}^{7}+{ \overline{w}} (  ( 5{{ \overline{w}}}^{3}{z}^{2}
+{ \overline{w}}{{ \overline{z}}}^{2}{w}^{2} ) {a}^{3}+2{ \overline{w}} ( 
{{ \overline{w}}}^{2}+{w}^{2} )  ( w{ \overline{w}}+5z{ \overline{z}} ) 
{a}^{2}\\
+ ( 5{{ \overline{w}}}^{3}{{ \overline{z}}}^{2}+ ( 8{w}^{2}{{
 \overline{z}}}^{2}+9{w}^{2}{z}^{2} ) { \overline{w}} ) a+8{w}^{5}+
16{{ \overline{w}}}^{2}{w}^{3}+8{ \overline{w}}{w}^{2}z{ \overline{z}} ) {b}^{6
}\\
+ (  ( -8wz{{ \overline{w}}}^{3}{ \overline{z}}+5{{ \overline{w}}}^{2}{{ 
\overline{z}}}^{2}{z}^{2} ) {a}^{4}+ ( 2w ( -3{{ \overline{z}}}^{2}
+{z}^{2} ) {{ \overline{w}}}^{3}+ ( 10{{ \overline{z}}}^{3}z+10{ 
\overline{z}}{z}^{3} ) {{ \overline{w}}}^{2}-8{w}^{3}{{ \overline{z}}}^{2}{ \overline{w}}
 ) {a}^{3}\\
 + ( {{ \overline{w}}}^{6}+2{w}^{2}{{ \overline{w}}}^{4}+32w
z{{ \overline{w}}}^{3}{ \overline{z}}+ ( 5{{ \overline{z}}}^{4}+20{z}^{2}{{ \overline{z}}
}^{2}+5{z}^{4} ) {{ \overline{w}}}^{2}+8{w}^{3}z{ \overline{w}}{ \overline{z}}
 ) {a}^{2}\\
 + (  ( 26{{ \overline{z}}}^{2}w+18w{z}^{2}
 ) {{ \overline{w}}}^{3}+ ( 10{{ \overline{z}}}^{3}z+10{ \overline{z}}{z}^
{3} ) {{ \overline{w}}}^{2}+ ( 24{{ \overline{z}}}^{2}{w}^{3}+16{w}^{
3}{z}^{2} ) { \overline{w}} ) a+24{w}^{2}{{ \overline{w}}}^{4}\\
+16wz{
{ \overline{w}}}^{3}{ \overline{z}}+ ( 5{z}^{2}{{ \overline{z}}}^{2}+32{w}^{4}
 ) {{ \overline{w}}}^{2}+24{w}^{3}z{ \overline{w}}{ \overline{z}}+4{w}^{6}
 ) {b}^{5}+ (  ( 5{w}^{2}{{ \overline{w}}}^{2}{{ \overline{z}}}^{2
}-8wz{ \overline{w}}{{ \overline{z}}}^{3} ) {a}^{5}\\
-8 ( z{{ \overline{w}}}
^{3}+{w}^{2}z{ \overline{w}}+w{ \overline{z}} ( {{ \overline{z}}}^{2}+{z}^{2}
 )  ) { \overline{w}}{ \overline{z}}{a}^{4}+ (  ( -7{{
 \overline{z}}}^{2}-18{z}^{2} ) {{ \overline{w}}}^{4}+ ( -44{w}^{2}{{
 \overline{z}}}^{2}-9{w}^{2}{z}^{2} ) {{ \overline{w}}}^{2}\\
 +8wz{ \overline{z}}
 ( {{ \overline{z}}}^{2}+{z}^{2} ) { \overline{w}}-9{w}^{4}{{ \overline{z}}}^{
2} ) {a}^{3}+ ( -2w{{ \overline{w}}}^{5}-18z{{ \overline{w}}}^{4}{
 \overline{z}}-6{w}^{3}{{ \overline{w}}}^{3}-6{w}^{2}z{{ \overline{w}}}^{2}{ \overline{z}}\\
 -2
w ( -8{{ \overline{z}}}^{4}-20{z}^{2}{{ \overline{z}}}^{2}+{w}^{4}-4{z}^{
4} ) { \overline{w}}-2{w}^{4}z{ \overline{z}} ) {a}^{2}+ ( 
 ( -2{{ \overline{z}}}^{2}+9{z}^{2} ) {{ \overline{w}}}^{4}+30{w}^{
2} ( 2{{ \overline{z}}}^{2}+{z}^{2} ) {{ \overline{w}}}^{2}\\
+ ( 32
{{ \overline{z}}}^{3}wz+24{ \overline{z}}w{z}^{3} ) { \overline{w}}+7{w}^{4}{
z}^{2}+16{w}^{4}{{ \overline{z}}}^{2} ) a+16w{{ \overline{w}}}^{5}+8z{{
 \overline{w}}}^{4}{ \overline{z}}+48{w}^{3}{{ \overline{w}}}^{3}+56{w}^{2}z{{ \overline{w}}}^
{2}{ \overline{z}}\\
+ ( 16{{ \overline{z}}}^{2}w{z}^{2}+16{w}^{5} ) {
 \overline{w}}+16{w}^{4}z{ \overline{z}} ) {b}^{4}+ (  ( {{ \overline{w}}
}^{2}{{ \overline{z}}}^{2}{z}^{2}+5{w}^{2}{{ \overline{z}}}^{4} ) {a}^{6}+
 ( 10w{{ \overline{w}}}^{3}{{ \overline{z}}}^{2}\\
 + ( -6{{ \overline{z}}}^{3}z-
8{ \overline{z}}{z}^{3} ) {{ \overline{w}}}^{2}+10{w}^{3}{{ \overline{z}}}^{2}{
 \overline{w}}+2{w}^{2}z{{ \overline{z}}}^{3} ) {a}^{5}+ ( 8wz{{ 
\overline{w}}}^{3}{ \overline{z}}+ ( -7{{ \overline{z}}}^{4}-44{z}^{2}{{ \overline{z}}}^{2}-
9{z}^{4} ) {{ \overline{w}}}^{2}\\
+8{w}^{3}z{ \overline{w}}{ \overline{z}}-9{w}
^{2}{{ \overline{z}}}^{2} ( 2{{ \overline{z}}}^{2}+{z}^{2} )  ) {
a}^{4}+ (  ( -40{{ \overline{z}}}^{2}w-24w{z}^{2} ) {{
 \overline{w}}}^{3}+ ( -40{{ \overline{z}}}^{3}z-40{ \overline{z}}{z}^{3}
 ) {{ \overline{w}}}^{2}\\
 -4{w}^{3} ( 10{{ \overline{z}}}^{2}+{z}^{2}
 ) { \overline{w}}-4{w}^{2}z{ \overline{z}} ( 6{{ \overline{z}}}^{2}+{z}^{
2} )  ) {a}^{3}+ ( -2{{ \overline{w}}}^{6}-12{w}^{2}{{
 \overline{w}}}^{4}-72wz{{ \overline{w}}}^{3}{ \overline{z}}\\
 + ( -4{{ \overline{z}}}^{4}-
29{z}^{2}{{ \overline{z}}}^{2}-10{w}^{4}-2{z}^{4} ) {{ \overline{w}}}^{2
}-32{w}^{3}z{ \overline{w}}{ \overline{z}}-{w}^{2} ( -15{{ \overline{z}}}^{4}-6
{z}^{2}{{ \overline{z}}}^{2}+{w}^{4}-2{z}^{4} )  ) {a}^{2}\\
+
\end{multline*}
\begin{multline*}
+ (  ( 18{{ \overline{z}}}^{2}w+12w{z}^{2} ) {{ \overline{w}}}^{3
}+ ( 2{{ \overline{z}}}^{3}z+4{ \overline{z}}{z}^{3} ) {{ \overline{w}}}^{
2}+ ( 38{{ \overline{z}}}^{2}{w}^{3}+12{w}^{3}{z}^{2} ) { 
\overline{w}}+12{w}^{2}{z}^{3}{ \overline{z}}+30{w}^{2}z{{ \overline{z}}}^{3} ) a\\
+4
{{ \overline{w}}}^{6}+32{w}^{2}{{ \overline{w}}}^{4}+40wz{{ \overline{w}}}^{3}{ \overline{z}
}+ ( 6{z}^{2}{{ \overline{z}}}^{2}+24{w}^{4} ) {{ \overline{w}}}^{2}+
40{w}^{3}z{ \overline{w}}{ \overline{z}}+15{w}^{2}{z}^{2}{{ \overline{z}}}^{2}
 ) {b}^{3}\\
 + ( {a}^{7}{z}^{2}{{ \overline{z}}}^{4}+10 ( {{
 \overline{z}}}^{2}+{z}^{2} )  ( w{ \overline{w}}+1/5z{ \overline{z}} ) 
{{ \overline{z}}}^{2}{a}^{6}+ ( 5{{ \overline{w}}}^{4}{{ \overline{z}}}^{2}+20{w}^
{2}{{ \overline{w}}}^{2}{{ \overline{z}}}^{2}+8wz{ \overline{z}} ( 4{{ \overline{z}}}^{
2}+{z}^{2} ) { \overline{w}}\\
+5{w}^{4}{{ \overline{z}}}^{2}+2{z}^{2}{{ 
\overline{z}}}^{4}+{{ \overline{z}}}^{6} ) {a}^{5}+ ( 16z{{ \overline{w}}}^{4}{
 \overline{z}}+40{w}^{2}z{{ \overline{w}}}^{2}{ \overline{z}}-2w ( 9{{ \overline{z}}}^{
4}+3{z}^{2}{{ \overline{z}}}^{2}+{z}^{4} ) { \overline{w}}-2z{{ \overline{z}}}^{5
}\\
-6{z}^{3}{{ \overline{z}}}^{3}+ ( 8{w}^{4}z-2{z}^{5} ) {
 \overline{z}} ) {a}^{4}+ (  ( -4{{ \overline{z}}}^{2}+15{z}^{2}
 ) {{ \overline{w}}}^{4}+ ( -29{w}^{2}{{ \overline{z}}}^{2}+6{w}^{2}{
z}^{2} ) {{ \overline{w}}}^{2}\\
+ ( -72{{ \overline{z}}}^{3}wz-32{ 
\overline{z}}w{z}^{3} ) { \overline{w}}-2{{ \overline{z}}}^{6}-12{z}^{2}{{ \overline{z}}}
^{4}+ ( -2{w}^{4}-10{z}^{4} ) {{ \overline{z}}}^{2}+2{w}^{4}
{z}^{2}-{z}^{6} ) {a}^{3}\\
+ ( -6w{{ \overline{w}}}^{5}+2z{{ 
\overline{w}}}^{4}{ \overline{z}}-14{w}^{3}{{ \overline{w}}}^{3}-54{w}^{2}z{{ \overline{w}}}^{2}{
 \overline{z}}-4 ( {w}^{4}+{z}^{4}+{\frac {27}{2}}{z}^{2}{{ \overline{z}}}^
{2}-1/2{{ \overline{z}}}^{4} ) w{ \overline{w}}\\
-4 ( {w}^{4}+{z}^{4}+7
/2{z}^{2}{{ \overline{z}}}^{2}+3/2{{ \overline{z}}}^{4} ) z{ \overline{z}}
 ) {a}^{2}+ (  ( 12{{ \overline{z}}}^{2}-2{z}^{2}
 ) {{ \overline{w}}}^{4}+3{w}^{2} ( 6{{ \overline{z}}}^{2}+{z}^{2}
 ) {{ \overline{w}}}^{2}\\
 +16wz{ \overline{w}}{{ \overline{z}}}^{3}-{{ \overline{z}}}^{2}
 ( 6{z}^{2}{{ \overline{z}}}^{2}+{w}^{4}+5{z}^{4} )  ) a
+16 ( w{ \overline{w}}+z{ \overline{z}} )  ( {{ \overline{w}}}^{2}{w}^{2}
+z{ \overline{w}}{ \overline{z}}w-1/8{z}^{2}{{ \overline{z}}}^{2}\\
+1/2{{ \overline{w}}}^{4}
 )  ) {b}^{2}+ (  (  ( 5{{ \overline{z}}}^{4}+8
{z}^{2}{{ \overline{z}}}^{2} ) {{ \overline{w}}}^{2}+9{w}^{2}{z}^{2}{{ 
\overline{z}}}^{2} ) {a}^{6}+ ( 10w{{ \overline{w}}}^{3}{{ \overline{z}}}^{2}+
 ( 26{{ \overline{z}}}^{3}z+24{ \overline{z}}{z}^{3} ) {{ \overline{w}}}^{
2}\\
+10{w}^{3}{{ \overline{z}}}^{2}{ \overline{w}}+16{w}^{2}{z}^{3}{ \overline{z}}+18{w
}^{2}z{{ \overline{z}}}^{3} ) {a}^{5}+ ( 32wz{{ \overline{w}}}^{3}{ 
\overline{z}}+ ( -2{{ \overline{z}}}^{4}+60{z}^{2}{{ \overline{z}}}^{2}+16{z}^{4}
 ) {{ \overline{w}}}^{2}\\
 +24{w}^{3}z{ \overline{w}}{ \overline{z}}+7{w}^{2}
 ( {z}^{4}+{\frac {30}{7}}{z}^{2}{{ \overline{z}}}^{2}+{\frac {9}{7}}
{{ \overline{z}}}^{4} )  ) {a}^{4}+ (  ( 2{{ \overline{z}}
}^{2}w+30w{z}^{2} ) {{ \overline{w}}}^{3}+ ( 18{{ \overline{z}}}^{3}z
+38{ \overline{z}}{z}^{3} ) {{ \overline{w}}}^{2}\\
+ ( 4{{ \overline{z}}}^{2
}{w}^{3}+12{w}^{3}{z}^{2} ) { \overline{w}}+12{w}^{2}z{ \overline{z}}
 ( {{ \overline{z}}}^{2}+{z}^{2} )  ) {a}^{3}+ ( -6{w
}^{2}{{ \overline{w}}}^{4}+16wz{{ \overline{w}}}^{3}{ \overline{z}}+ ( 12{{ \overline{z}}
}^{4}\\
+18{z}^{2}{{ \overline{z}}}^{2}-5{w}^{4}-{z}^{4} ) {{ \overline{w}}}^
{2}+3{w}^{2}{z}^{2}{{ \overline{z}}}^{2}-2{w}^{2}{{ \overline{z}}}^{4} ) {
a}^{2}-2 ( w{ \overline{w}}+z{ \overline{z}} )  (  ( -8{{
 \overline{z}}}^{2}+{z}^{2} ) {{ \overline{w}}}^{2}+{w}^{2}{{ \overline{z}}}^{2}
 ) a\\
 +4{{ \overline{w}}}^{2} ( w{ \overline{w}}+z{ \overline{z}} ) ^{2}
 ) b+4 ( {a}^{6}{z}^{4}{{ \overline{z}}}^{2}+ ( 2{ \overline{w}}
{{ \overline{z}}}^{2}w{z}^{2}+4{z}^{3}{{ \overline{z}}}^{3}+2{z}^{5}{ \overline{z}}
 ) {a}^{5}+ ( 5/4{w}^{2}{{ \overline{w}}}^{2}{{ \overline{z}}}^{2}\\
 +
 ( 4{{ \overline{z}}}^{3}wz+6{ \overline{z}}w{z}^{3} ) { \overline{w}}+{z}
^{2} ( 6{{ \overline{z}}}^{4}+8{z}^{2}{{ \overline{z}}}^{2}+{z}^{4}
 )  ) {a}^{4}+ ( 4{w}^{2}z{{ \overline{w}}}^{2}{ \overline{z}}+4
 ( {z}^{4}+7/2{z}^{2}{{ \overline{z}}}^{2}+1/2{{ \overline{z}}}^{4}
 ) w{ \overline{w}}\\
 +4{z}^{5}{ \overline{z}}+12{z}^{3}{{ \overline{z}}}^{3}+4z{{
 \overline{z}}}^{5} ) {a}^{3}+ ( {\frac {15}{4}}{w}^{2} ( {
z}^{2}+2/5{{ \overline{z}}}^{2} ) {{ \overline{w}}}^{2}+10wz{ \overline{z}}
 ( {{ \overline{z}}}^{2}+{z}^{2} ) { \overline{w}}+6{z}^{4}{{ \overline{z}}}^{
2}+8{z}^{2}{{ \overline{z}}}^{4}+{{ \overline{z}}}^{6} ) {a}^{2}\\
-1/2
 ( w{ \overline{w}}+z{ \overline{z}} )  ( {{ \overline{w}}}^{2}{w}^{2}-8z
{ \overline{w}}{ \overline{z}}w-4{{ \overline{z}}}^{4}-8{z}^{2}{{ \overline{z}}}^{2}
 ) a+{{ \overline{z}}}^{2} ( w{ \overline{w}}+z{ \overline{z}} ) ^{2}
 ) a \Bigr],
\end{multline*}
\end{footnotesize}

This formula for $H$ (restricted on $M$) was computed by employing the MAPLE package. As a consequence of the formula and the fact that $H=0$ is the condition for sphericity of a hypersurface $M$ (see e.g. the discussion in  \cite{nonanalytic,KZ2} for details), we obtain the following result

\begin{lemma}\label{spherical}
An ellipsoid $M$ as above is never spherical, unless $A_1=A_2=0$ holds.
\end{lemma}
\begin{proof}
We analyze the numerator of the above expression for $H$, considering it as a polynomial in $\bar w$ (with coefficients being polynomials in $z,\bar z,w$), and compute the remainder  for its division by $\rho$ (again considering the latter as a polynomial in $\bar w$). 
The latter remainder (which has to vanish identically in the case of identical vanishing of $H$ on $M$) has the lowest degree term (with $A_1=a, A_2=b$)
\begin{eqnarray}\label{leading}
&&\quad {(96a^3-192a)b^3-48a^3b^5\over b^4}\\
&&\quad+\frac{(264 a^4+96 a^2-96) b^5
+(-120 a^6-96 a^4+96 a^2) b^3-144 a^2 b^7} {b^4}\overline{z}^2 \nonumber \\
&&\quad+\frac{(-552 a^4+432 a^2) b^3-144 a^4 b^7+(576 a^4-312 a^2) b^5}{b^4}z^2.\nonumber
\end{eqnarray} 
It is straightforward to check that \eqref{leading} vanishes identically if and only if $a=b=0$, since $A_1=a$ and $A_2=b$,  this proves the lemma.
\end{proof}

Finally, we have to compute the obstruction function $\Omega^4_{2,0}=(\partial_p)^2H$ (using the MAPLE package), but due to its length we will not present the full outcome here. Arguing as in the proof of Lemma \ref{spherical}, we see that it's sufficient for our purposes to consider the numerator of the rational expression for $\Omega^4_{2,0}$, and analyze the remainder for its division by $\rho$, where both the numerator and the defining function $\rho$ are considered as polynomials in $\overline{w}$. Assuming the identical vanishing of the obstruction function $\Omega^4_{2,0}$, we obtain the identical vanishing of the remainder. Considering the lowest degree terms in $\bar w$ for the latter one, we get with $A_1=a$ and $A_2=b$:
\begin{eqnarray}\label{latter}
{2\over b^2} f_1(a,b) w+{2\over b} f_2(a,b) \overline{w} +{a\over b}  f_3(a,b)z^2\overline{w}=0,
\end{eqnarray}
where
{\small\smaller{
\begin{eqnarray*}
f_1(a,b)&=&-84a^2b^6-(-165a^4-120a^2+72)b^4\\
&&-(90a^6+174a^4+4a^2-64)b^2-(-80a^6-24a^4+32a^2),
\end{eqnarray*}
$$ f_2(a, b)=(10a^6-21a^4b^2+12a^2b^4+6a^4-8a^2b^2-4a^2+8b^2)
$$
and
\begin{eqnarray*}
f_3(a, b)&=& 135a^6b^2-198a^4b^4+60a^2b^6-280a^6+453a^4b^2\\
&&-156a^2b^4+57a^4-154a^2b^2+60b^4+64a^2-32b^2-12.
\end{eqnarray*}
\normalsize
The equation (\ref{latter}) can hold true only if $a=b=0$. This completely proves Theorem \ref{main} and  the Ramadanov Conjecture for ellipsoids in $\CC^{2}$.

\bigskip

\begin{center}
\bf Acknowledgements
\end{center}

\bigskip

The first author was supported by the  FWF grant 10.55776/PAT4819724. The second author was supported by the FWF grant P34369, and the NSFC grant K24281006.

 \end{document}